\documentclass[a4paper]{amsart}
\usepackage{amsfonts,amsmath,amsthm,amssymb,color}
\usepackage[all]{xy}

\usepackage[ colorlinks=true, linkcolor=blue, citecolor=blue]{hyperref}

\numberwithin{equation}{section}

\begin{document}

\newtheorem{theorem}{Theorem}[section]
\newtheorem{lemma}[theorem]{Lemma}
\newtheorem{proposition}[theorem]{Proposition}
\newtheorem{corollary}[theorem]{Corollary}
\newtheorem{conjecture}{Conjecture}

\theoremstyle{definition}
\newtheorem{definition}[theorem]{Definition}
\newtheorem{example}[theorem]{Example}

\theoremstyle{remark}
\newtheorem{remark}[theorem]{Remark}
\newtheorem*{ack}{Acknowledgments}

\newenvironment{magarray}[1]
{\renewcommand\arraystretch{#1}}
{\renewcommand\arraystretch{1}}

\newcommand{\mapor}[1]{\smash{\mathop{\longrightarrow}\limits^{#1}}}
\newcommand{\mapin}[1]{\smash{\mathop{\hookrightarrow}\limits^{#1}}}
\newcommand{\mapver}[1]{\Big\downarrow
\rlap{$\vcenter{\hbox{$\scriptstyle#1$}}$}}
\newcommand{\liminv}{\smash{\mathop{\lim}\limits_{\leftarrow}\,}}

\newcommand{\Set}{\mathbf{Set}}
\newcommand{\Art}{\mathbf{Art}}
\newcommand{\solose}{\Rightarrow}

\newcommand{\specif}[2]{\left\{#1\,\left|\, #2\right. \,\right\}}

\renewcommand{\bar}{\overline}
\newcommand{\de}{\partial}
\newcommand{\debar}{{\overline{\partial}}}
\newcommand{\per}{\!\cdot\!}
\newcommand{\Oh}{\mathcal{O}}
\newcommand{\sA}{\mathcal{A}}
\newcommand{\sB}{\mathcal{B}}
\newcommand{\sC}{\mathcal{C}}
\newcommand{\sD}{\mathcal{D}}
\newcommand{\sE}{\mathcal{E}}
\newcommand{\sF}{\mathcal{F}}\newcommand{\sG}{\mathcal{G}}
\newcommand{\sH}{\mathcal{H}}
\newcommand{\sI}{\mathcal{I}}
\newcommand{\sJ}{\mathcal{J}}
\newcommand{\sL}{\mathcal{L}}
\newcommand{\sM}{\mathcal{M}}
\newcommand{\sP}{\mathcal{P}}
\newcommand{\sU}{\mathcal{U}}
\newcommand{\sV}{\mathcal{V}}
\newcommand{\sX}{\mathcal{X}}
\newcommand{\sY}{\mathcal{Y}}
\newcommand{\sN}{\mathcal{N}}
\newcommand{\sZ}{\mathcal{Z}}

\newcommand{\Aut}{\operatorname{Aut}}
\newcommand{\Mor}{\operatorname{Mor}}
\newcommand{\Def}{\operatorname{Def}}
\newcommand{\Hom}{\operatorname{Hom}}
\newcommand{\Hilb}{\operatorname{Hilb}}
\newcommand{\HOM}{\operatorname{\mathcal H}\!\!om}
\newcommand{\DER}{\operatorname{\mathcal D}\!er}
\newcommand{\Spec}{\operatorname{Spec}}
\newcommand{\Der}{\operatorname{Der}}
\newcommand{\End}{{\operatorname{End}}}
\newcommand{\END}{\operatorname{\mathcal E}\!\!nd}
\newcommand{\Image}{\operatorname{Im}}
\newcommand{\coker}{\operatorname{coker}}
\newcommand{\tot}{\operatorname{tot}}
\newcommand{\Diff}{\operatorname{Diff}}
\newcommand{\ten}{\otimes}
\newcommand{\mA}{\mathfrak{m}_{A}}

\renewcommand{\Hat}[1]{\widehat{#1}}
\newcommand{\dual}{^{\vee}}
\newcommand{\desude}[2]{\dfrac{\de #1}{\de #2}}
\newcommand{\sK}{\mathcal{K}}
\newcommand{\A}{\mathbb{A}}
\newcommand{\N}{\mathbb{N}}
\newcommand{\R}{\mathbb{R}}
\newcommand{\Z}{\mathbb{Z}}
\renewcommand{\H}{\mathbb{H}}
\renewcommand{\L}{\mathbb{L}}
\newcommand{\proj}{\mathbb{P}}
\newcommand{\K}{\mathbb{K}\,}
\newcommand\C{\mathbb{C}}
\newcommand\T{\mathbb{T}}
\newcommand\Del{\operatorname{Del}}
\newcommand\Tot{\operatorname{Tot}}
\newcommand\Grpd{\mbox{\bf Grpd}}
\newcommand\rif{~\ref}

\newcommand\vr{``}
\newcommand{\rh}{\rightarrow}
\newcommand{\contr}{{\mspace{1mu}\lrcorner\mspace{1.5mu}}}

\newcommand{\bi}{\boldsymbol{i}}
\newcommand{\bl}{\boldsymbol{l}}
\newcommand{\op}{\operatorname}
\newcommand{\MC}{\operatorname{MC}}
\newcommand{\Coder}{\operatorname{Coder}}
\newcommand{\TW}{\operatorname{TW}}
\newcommand{\id}{\operatorname{id}}
\newcommand{\ad}{\operatorname{ad}}
\newcommand{\cone}{\operatorname{C}}
\newcommand{\cylinder}{\operatorname{Cyl}}
\newcommand{\tr}{\triangleright}
\newcommand{\tl}{\triangleleft}

\title{A trick to compute certain Manin products of operads}
\author{Ruggero Bandiera}
\begin{abstract} We describe simple tricks to compute the Manin black products with the operads $\mathcal{A}ss$, $\mathcal{C}om$ and $pre\mathcal{L}ie$. \end{abstract}

\maketitle
\section*{Introduction}

The aim of this paper is to describe a simple method to compute the Manin black products with the operads $\sA ss$, $\sC om$ and $pre\sL ie$. For instance, this allows the computation of $\sA ss\bullet\sC om$ and $pre\sL ie\bullet pre\sL ie$ (Examples \ref{ex:ass} and \ref{ex:prelie}), answering some questions posed by Loday \cite{L}  (these had beeen answered already \cite{bai2,GK}). While simple methods to compute $pre\sL ie\bullet-$ were already known \cite{bai,bai3}, to our knowledge the results are new for the functors $\sA ss\bullet-$, $\sC om\bullet-$. We illustrate the method by several explicit computations: among these, we show that in the square diagram of operads introduced by Chapton \cite{cha}   
\[\xymatrix{ pre\sL ie\ar[r]&\sD end\ar[r]& \sZ inb\\ \sL ie\ar[r]\ar[u]&\sA ss\ar[r]\ar[u]&\sC om\ar[u]\\\sL eib\ar[r]\ar[u]&di\sA ss\ar[r]\ar[u]&\sP erm\ar[u] }\]
where it is well known that the top row is the Manin black product of the middle one with $pre\sL ie$ and the bottom row is the Manin white product of the middle one with $\sP erm$, it is also true that the top row is the Manin \emph{white} product of the middle one with $\sZ inb$ and the bottom row is the Manin \emph{black} product of the middle one with $\sL eib$.

We explain the method by considering the case of $(\sA ss,\cup)$, where we denote by $\cup$ the generating associative product. It is convenient to consider first the (Koszul) dual computation of the Manin white
product $\sA ss\circ-$. Roughly, given an operad $\Oh$, which will be always an operad in vector spaces over a field $\mathbb{K}$ and moreover binary, quadratic and finitely generated by non-symmetric operations $\cdot_i$, $i=1,\ldots,p$, symmetric operations $\bullet_j$, $j=1,\ldots,q$, and anti-symmetric operations $[-,-]_k$, $k=1,\ldots,r$, the operad $\sA ss\circ\Oh$ is generated by the tensor product operations $\cup\otimes\cdot_i$, $\cup\otimes\cdot_i^{op}$, $\cup\otimes\bullet_j$, $\cup\otimes[-,-]_k$ (all of which are non-symmetric, and where $\cdot_i^{op}$ are the opposite products) together with the relations holding in the tensor product $A\ten V$ of a generic $\sA ss$-algebra $(A,\cup)$ and a generic $\Oh$-algebra $(V,\cdot_i,\bullet_j,\ast_k)$. We define a functor $\op{Ass}_\circ(-):\mathbf{Op}\to\mathbf{Op}$ by replacing the generic $\sA ss$-algebra $(A,\cup)$ in the above definition of $\sA ss\circ-$ with the dg associative algebra $(C^*(\Delta_1;\mathbb{K}),\cup)$ of non-degenerate cochains on the $1$-simplex with the usual cup product, and show that $\op{Ass}_\circ(\Oh)=\sA ss\circ\Oh$, essentially because the only relations satisfied in $(C^*(\Delta_1;\mathbb{K}),\cup)$ are the associativity relations. 

We notice that besides the relation of $\sA ss\circ\Oh$-algebra, the tensor product operations (which we will simply call the cup products in the body of the paper) on $C^*(\Delta_1;V)=C^*(\Delta_1;\mathbb{K})\otimes V$ satisfy the Leibniz identity with respect to the differential: moreover, for any $X\subset\Delta_1$ (that is, the boundary or one of the vertices) the subcomplex of relative cochains $C^*(\Delta_1,X;V)\subset C^*(\Delta_1;V)$ is a dg $\sA ss\circ\Oh$-ideal. We reassume the above properties by saying that $C^*(\Delta_1;V)$ with the tensor product operations (cup products) is a local dg $\sA ss\circ\Oh$-algebra. 

Conversely, our trick to compute $\sA ss\bullet\Oh$ consists in imposing a local dg $\Oh$-algebra structure $(C^*(\Delta_1;V),\cdot_i,\bullet_j,[-,-]_k)$ on the complex $C^*(\Delta_1;V)$. Notice that the space $C^*(\Delta_1;V)$ splits into the direct sum of three copies of $V$, we write $C^*(\Delta_1;V)=V_0\oplus V_1\oplus V_{01}$, where $V_0$ (resp.: $V_1$) is the copy corresponding to the left (resp.: right) vertex and $V_{01}$ is the copy corresponding to the $1$-dimensional cell. The locality assumption and the Leibniz relation with respect to the differential imply that the whole $\Oh$-algebra structure on $C^*(\Delta_1;V)$ is determined by the products $\cdot_i:V_0\otimes V_{01}\to V_{01}$, $\cdot_i:V_{01}\otimes V_0\to V_{01}$, $\bullet_j:V_0\otimes V_{01}\to V_{01}$, $[-,-]_k:V_0\otimes V_{01}\to V_{01}$, that is, by the datum of $2p+q+r$ non-symmetric operations $\prec_i,\succ_i,\circ_j,\ast_k$ on $V$: then the relations of $\Oh$-algebra on $C^*(\Delta_1;V)$ are equivalent to certain relations on the products $\prec_i,\succ_i,\circ_j,\ast_k$, and this defines a functor $\op{Ass}_\bullet(-):\mathbf{Op}\to\mathbf{Op}$. We remark that the computation of this functor is completely mechanical. Finally, this is exactly the same as the functor $\sA ss\bullet-$: this will be proved at the very end of the paper, by showing that the functors $\xymatrix{\op{Ass}_\bullet(-):\mathbf{Op}\ar@<2pt>[r]& \mathbf{Op}:\op{Ass}_\circ(-)\ar@<2pt>[l]}$ form an adjoint pair (notice how the counit is obvious: given an $\Oh$-algebra structure on $V$, there is a local dg $\op{Ass}_\circ(\Oh)$-algebra structure on $C^*(\Delta_1;V)$ via the tensor product operations, and by defninition this is the same as an $\op{Ass}_\bullet(\op{Ass}_\circ(\Oh))$-algebra structure on $V$).

The trick to compute $\sC om\bullet-$ and $pre\sL ie\bullet-$ (and dually $\sL ie\circ-$ and $\sP erm\circ-$) is similar: we replace the dg associative algebra $(C^*(\Delta_1;\mathbb{K}),\cup)$ in the above discussion with the dg Lie algebra $(C^*(\Delta_1;\mathbb{K}),[-,-]=\cup-\cup^{op})$ in the first case, and with the dg (right) permutative algebra $(C^*(\Delta_1,v_l;\mathbb{K}),\cup)$ (where $v_l\subset\Delta_1$ is the left vertex) in the second case.

\begin{ack} The author is grateful to Domenico Fiorenza for some useful discussions.
\end{ack}

\bigskip

\subsection*{Preliminary remarks.} The author is not an actual expert on operads (and is afraid this might be painfully evident throughout the reading to anyone who is), accordingly, we will think of operads n\"aively, as defined by the corresponding type of algebras. Moreover, although the constructions seem to make sense in more general settings (for instance, dg operads), we limit ourselves to work in the category  $\mathbf{Op}$ of finitely generated binary quadratic symmetric operads on vector spaces over a field $\mathbb{K}$.

Given a vector space $V$, we denote by $V^{\ten n}$, $V^{\odot n}$, $V^{\wedge n}$, $n\geq0$, the tensor powers, symmetric powers and exterior powers of $V$ respectively. We shall always denote by $(\Oh,\cdot_i,\bullet_j,[-,-]_k)$ an operad in $\mathbf{Op}$ generated by non-symmetric products $\cdot_i$, $i=1,\ldots,p$, commutative products~$\bullet_j$, $j=1,\ldots,q$, and anti-commutative brackets $[-,-]_k$, $k=1,\ldots,r$: the relations shall be omitted from the notation. For the definitions of the Koszul duality functor $-^!:\mathbf{Op}\to\mathbf{Op}$ and the Manin white and black products $-\circ-,-\bullet-:\mathbf{Op}\times\mathbf{Op}\to\mathbf{Op}$ we refer to \cite{V,LV}. For the definitions of the various operads we will consider, where not already specified, we refer to \cite{LV,zinb}.

We denote by $\Delta_1$ the standard $1$-simplex, which shall be represented as an arrow $\to$, by $v_i\subset\Delta_1$, $i=l,r$, the left and right vertex respectively, and by $\de\Delta_1=v_l\sqcup v_r\subset\Delta_1$ the boundary. Given a vector space $V$, we shall denote by $C^\ast(\Delta_1;V)$ (resp.: $C^*(\Delta_1\de\Delta_1;V)$, $C^*(\Delta_1,v_i;V)$, $i=l,r$) the usual complex of non-degenerate (resp.: relative) cochains with coefficients in $V$. We shall depict a $0$-cochain in $C^*(\Delta_1;V)$ as $_x\to_y$, $x,y\in V$ (moreover, we write $_x\to$ for $_x\to_0$ and $\to_y$ for $_0\to_y$), and a $1$-cochain as $\xrightarrow{x}$, similarly in the relative cases. The differential $d:C^0(\Delta_1;V)\to C^1(\Delta_1;V)$ is the usual one $d(_x\to_y)=\xrightarrow{y-x}$.
\section{A trick to compute Manin black products}

In this section we shall introduce, and compute in several cases, endofunctors $\mathbf{Op}\to\mathbf{Op}$ which we denote by  $\op{Ass}_\bullet(-)$, $\op{Com}_\bullet(-)$,  $\op{preLie}_\bullet(-)$, later we shall prove that these coincide with the functors  $\sA ss\bullet-$, $\sC om\bullet-$, $pre\sL ie\bullet-$ respectively.



\begin{definition}\label{def:ManinBlack} Given an operad $(\Oh,\cdot_i,\bullet_j,[-,-]_k)$ in $\mathbf{Op}$ and a vector space $V$, a structure of $\op{Ass}_\bullet(\Oh)$-algebra on $V$ is the datum of operations 
\begin{multline*} \cdot'_i:C^\ast(\Delta_1;V)^{\ten 2}\to C^*(\Delta_1;V),\quad\bullet'_j:C^\ast(\Delta_1;V)^{\odot2}\to C^*(\Delta_1;V)\\\mbox{and}\quad[-,-]'_k:C^\ast(\Delta_1;V)^{\wedge2}\to C^*(\Delta_1;V)\quad\mbox{such that}\end{multline*}
\begin{enumerate}
\item\label{item:dg} $(C^*(\Delta_1;V),d,\cdot'_i,\bullet'_j,[-,-]'_k)$ is a dg $\Oh$-algebra structure on $C^*(\Delta_1;V)$, and moreover the following locality assumption holds:\\

\item\label{item:locality} for all closed subsets $X\subset\Delta_1$ (i.e., $X=v_l,v_r,\de\Delta_1$, cf. the preliminary remarks) the complex of relative cochains $C^*(\Delta_1,X;V)\subset C^*(\Delta_1;V)$ is a dg $\Oh$-ideal.
\end{enumerate}
\end{definition}
The functor $\op{preLie}_\bullet(-)$ is defined similarly. 
\begin{definition} A $\op{preLie}_\bullet(\Oh)$-algebra structure on $V$ is the datum of a dg $\Oh$-algebra structure $(C^*(\Delta_1,v_l;V),d,\cdot'_i,\bullet'_j,[-,-]'_k)$ on $C^*(\Delta_1,v_l;V)$ (we notice that in this case the locality assumption, that is, the fact that $C^*(\Delta_1,\de\Delta_1;V)\subset C^*(\Delta_1,v_l;V)$ is a dg $\Oh$-ideal, is satisfied for trivial degree reasons).
\end{definition}  
\begin{remark}\label{rem:prelie}
	This actually defines $\op{preLie}_\bullet(-)=\op{preLie}_{r,\bullet}(-)$, by replacing the vertex $v_l$ with the one $v_r$ in the previous definition we get a second endofunctor $\op{preLie}_{l,\bullet}(-)$ : more on this in Remark~\ref{rem:leftcase}.
\end{remark} 
\newcommand{\w}{\widetilde}

It is not immediately obvious that this defines endofunctors $\mathbf{Op}\to\mathbf{Op}$: we consider first the case of $\op{Ass}_\bullet(-)$. Given the datum $(C^\ast(\Delta_1;V),d,\cdot_i',\bullet_j',[-,-]_k')$ of a local dg $\Oh$-algebra structure on $(C^\ast(\Delta_1;V),d)$, we notice that the locality assumption implies $_x\to\cdot_i'\to_y\:=\:\to_x\cdot_i'\,\,_y\!\to=0$, thus, applying the differential $d$ and Leibniz rule we see that $_x\to\cdot_i'\xrightarrow{y}\:=\:\xrightarrow{x}\cdot_i'\to_y$, and similarly $\to_x\cdot_i'\xrightarrow{y}\:=\:\xrightarrow{x}\cdot_i'\,\,_y\to$. We define non-symmetric products $\prec_i,\succ_i:V^{\ten 2}\to V$ on $V$ by 
\[ _x\to\cdot_i'\xrightarrow{y}\,\,=:\,\,\xrightarrow{x\prec_iy}\,\,:=\,\,\xrightarrow{x}\cdot_i'\to_y,\qquad\to_x\cdot_i'\xrightarrow{y}\,\,=:\,\,-\xrightarrow{y\succ_i x}\,\,:=\,\,\xrightarrow{x}\cdot_i'\,\,_y\to.\]
Always by locality and Leibniz rule we have a product $\cdot_i:V^{\ten2}\to V$ defined equivalently by $\to_x\cdot_i'\to_y=\to_{x\cdot_iy}$ or $_x\to\cdot_i'\,_y\to=_{x\cdot_i y}\to$, and moreover $x\cdot_i y= x\prec_i y - y\succ_i x$. In fact, 
\[\xrightarrow{x\cdot_iy}=d\left(\to_x\cdot_i'\to_y\right)=\xrightarrow{x}\cdot_i'\to_y+\to_x\cdot_i'\xrightarrow{y}=\xrightarrow{x\prec_iy-y\succ_ix}=-d\left(_x\to\cdot_i'\,_y\to\right).\]
In the same way, there are non-symmetric products $\circ_j,\ast_k:V^{\ten2}\to V$ on $V$, defined by the formulas 
\begin{multline*} _x\to\bullet_j'\xrightarrow{y}\,\,=\,\,\xrightarrow{x}\bullet_j'\to_y\,\,=\,\,\xrightarrow{y}\bullet_j'\,\,_x\to\,\,=\,\,\to_y\bullet_j'\xrightarrow{x}\,\,=:\,\,\xrightarrow{x\circ_j y},\\ \left[ _x\to,\xrightarrow{y}\right]_k'\,\,=\,\, \left[ \xrightarrow{x},\to_y\right]_k' \,\,=\,\, -\left[\xrightarrow{y},\,_x\to\right]_k' \,\,=\,\, -\left[ \to_y,\xrightarrow{x}\right]_k'\,\,=:\,\, \xrightarrow{x\ast_k y},\end{multline*}
products $\bullet_j:V^{\odot 2}\to V$, $[-,-]_k:V^{\wedge2}\to V$ defined by 
\[\to_x\bullet_j'\to_y=\to_{x\bullet_j y},\qquad_x\to\bullet_j'\,_y\to=_{x\bullet_jy}\to,\qquad \left[ \to_x,\to_y\right]_k'=\to_{[x,y]_k},\qquad\left[ _x\to,\,_y\to \right]_k'=_{[x,y]_k}\to, \]
and moreover by Leibniz rule
\[ x\bullet_j y = x\circ_j y + y \circ_j x, \qquad [x,y]_k=x\ast_k y - y\ast_k x.\]

\emph{Warning:} the previous formulas will be used extensively troughout the computations of this section without further mention.
 
In other words, the datum of an $\op{Ass}_\bullet(\Oh)$-algebra structure on $V$ is the same as the one of operations $\prec_i,\succ_i,\circ_j,\ast_k:V^{\ten 2}\to V$ on $V$, inducing operations $\cdot_i',\bullet_j',[-,-]_k'$ on $C^\ast(\Delta_1;V)$ via the previous formulas: the requirement that these make the latter into a (by construction, local dg) $\Oh$-algebra translates into a finite set of terniary relations on the operations $\prec_i,\succ_i,\circ_j,\ast_k$. More precisely, for every relation $R(x,y,z)=0$ satisfied in an $\Oh$-algebra, we get six (in general not independent) relations in the operad $\op{Ass}_\bullet(\Oh)$ as in the next lemma.

\begin{lemma}\label{lem:rel} Given operations $\prec_i,\succ_i,\circ_j,\ast_k$ on $V$ inducing operations $\cdot_i',\bullet_j',[-,-]_k'$ on $C^\ast(\Delta_1;V)$ as above, then the latter is a local dg $\Oh$-algebra if and only if for every relation $R(x,y,z)=0$ in the operad $\Oh$ the six relations  
\begin{multline*}0=R(\xrightarrow{x},\to_y,\,_z\to)=R(\xrightarrow{x},\,_y\to,\to_z)=R(\to_x,\xrightarrow{y},\,_z\to)=\\=R(_x\to,\xrightarrow{y},\to_z)=R(\to_x,\,_y\to,\xrightarrow{z})=R(_x\to,\to_y,\xrightarrow{z})=0,\end{multline*}
are satisfied.
\end{lemma}
\begin{proof} The only if part is clear. For the if part we have to show that the given relations imply all the others. First of all, we notice that the only relations not trivially satisfied are those in total degree 0 or 1. Moreover, in total degree 0 all necessary relations follow by the locality assumption but the ones $R(_x\to,\,_y\to,\,_z\to)=0=R(\to_x,\to_y,\to_z)$. In total degree 1, all the necessary relations are satisfied by hypothesis but the ones
\begin{multline*}0=R(_x\to,\,_y\to,\xrightarrow{z})=R(\to_x,\to_y,\xrightarrow{z})=R(_x\to,\xrightarrow{y},\,_z\to)=\\=R(\to_x,\xrightarrow{y},\to_z)=R(\xrightarrow{x},\,_y\to,\,_z\to)=R(\xrightarrow{x},\to_y,\to_z)=0.\end{multline*}
For instance, to prove the first one we apply $d$ to the relation $0=R(_x\to,\,_y\to,\to_z)$ (which follows by locality), and by Leibniz rule and the hypothesis of the lemma
\[0 = -R(\xrightarrow{x},\,_y\to,\to_z)-R(_x\to,\xrightarrow{y},\to_z)+R(_x\to,\,_y\to,\xrightarrow{z})=R(_x\to,\,_y\to,\xrightarrow{z}). \]
The others are proved similarly. Finally, if we denote by $R(x,y,z)$ the relation $R$ computed in the operations $\cdot_i$, $\bullet_j$, $[-,-]_k$ on $V$, we have $R(\to_x,\to_y,\to_z)=\to_{R(x,y,z)}$, and by Leibniz rule
\[\xrightarrow{R(x,y,z)}=d\left(\to_{R(x,y,z)}\right)=R(\xrightarrow{x},\to_y,\to_z)+R(\to_x,\xrightarrow{y},\to_z)+R(\to_x,\to_y,\xrightarrow{z})=0,\]
hence $R(\to_x,\to_y,\to_z)=0$, and $R(_x\to,\,_y\to,\,_z\to)=0$ is proved similarly.
\end{proof}
\begin{example}\label{ex:ass} We consider the operad $\sL ie$ of Lie algebras. By the previous argument,  an $\op{Ass}_\bullet(\sL ie)$-algebra structure on $V$ is the datum of $\ast:V^{\ten 2}\to V$ such that
\begin{multline*} \left[\left[_x\to,\to_y\right],\xrightarrow{z}\right]=0=\left[ _x\to,\left[\to_y,\xrightarrow{z}\right]\right]+\left[\left[_x\to,\xrightarrow{z}\right],\to_y\right]=\\=\left[_x\to,\xrightarrow{-z\ast y}\right]+\left[\xrightarrow{x\ast z},\to_y\right]=\xrightarrow{-x\ast(z\ast y)+(x\ast z)\ast y},\qquad\forall x,y,z\in V.
\end{multline*}
In other words, $\ast$ is an associative product on $V$. The other five relations to be checked in the previous lemma follow from this one by symmetry of the Jacobi identity, thus, we see that $\op{Ass}_\bullet(\sL ie)=\sA ss$, the operad of associative algebras.  \end{example}

We may repeat the previous considerations in the case of $\op{preLie}_\bullet(-)$, showing that a dg $\Oh$-algebra structure $(C^\ast(\Delta_1,v_l;V),d,\cdot_i'.\bullet_j',[-,-]_k')$ on the complex $C^\ast(\Delta_1,v_l;V)$ is determined, via the previous formulas, by operations $\prec_i,\succ_i,\circ_j,\ast_k:V^{\ten 2}\to V$ on $V$. The analog of Lemma \ref{lem:rel} is the following
\begin{lemma}\label{lem:relpl}○ Given operations $\prec_i,\succ_i,\circ_j,\ast_k$ on $V$ inducing operations $\cdot_i',\bullet_j',[-,-]_k'$ on $C^\ast(\Delta_1,v_l;V)$ as above, then the latter is a local dg $\Oh$-algebra if and only if for every relation $R(x,y,z)=0$ in the operad $\Oh$ the three relations  
\begin{equation*}R(\xrightarrow{x},\to_y,\to_z)=R(\to_x,\xrightarrow{y},\to_z)=R(\to_x,\to_y,\xrightarrow{z})=0,\end{equation*}
are satisfied. 
\end{lemma}
\begin{proof} It follows from the proof of Lemma \ref{lem:rel}.\end{proof}
\begin{example}\label{ex:prelie} A $\op{preLie}_\bullet(\sL ie)$-algebra structure on $V$ is the datum of $\ast:V^{\ten 2}\to V$ such that, where as usual $[x,y]=x\ast y - y \ast x$,
\begin{multline*} \left[\xrightarrow{x},\left[\to_y,\to_z\right]\right]=\left[\xrightarrow{x},\to_{[y,z]} \right]=\xrightarrow{x\ast[y,z]}=\\=\left[ \left[\xrightarrow{x},\to_y\right],\to_z\right]+\left[\to_y,\left[\xrightarrow{x},\to_z\right]\right]=\left[\xrightarrow{x\ast y},\to_z\right]+\left[\to_y,\xrightarrow{x\ast z}\right]=\xrightarrow{(x\ast y)\ast z- (x\ast z)\ast y},\qquad\forall x,y,z\in V.
\end{multline*}
In other words, $\ast$ is a right pre-Lie product on $V$. Since the remaining two relations to be checked in the previous lemma follow from this one by symmetry of the Jacobi identity, we see that $\op{preLie}_\bullet(\sL ie)=pre\sL ie$, the operad of (right) pre-Lie algebras. \end{example}
\begin{remark}\label{rem:leftcase} As was said in Remark \ref{rem:prelie} the above construction is actually the one of $\op{preLie}_\bullet(\sL ie)=\op{preLie}_{r,\bullet}(\sL ie)$: applying the functor $\op{preLie}_{l,\bullet}(-)$ introduced there we find by similar computations as in the above example that $\op{preLie}_{l,\bullet}(\sL ie)$ is the operad of \emph{left} pre-Lie algebras. More in general, it is easy to see that $\op{preLie}_{l,\bullet}(\Oh)$ is always the opposite of the operad $\op{preLie}_{r,\bullet}(\Oh)=\op{preLie}_{\bullet}(\Oh)$ for any $\Oh\in\mathbf{Op}$. We won't insist further on this point, and restrict our computations to the right case.
\end{remark}

By the proof of Lemma \ref{lem:rel} we have morphisms of operads $\Oh\to \op{Ass}_\bullet(\Oh)$, $\Oh\to \op{preLie}_\bullet(\Oh)$, sending an $\op{Ass}_\bullet(\Oh)$-algebra structure $(V,\prec_i,\succ_i,\circ_j,\ast_k)$ on $V$ to the associated $\Oh$-algebra structure \[(V\:\:,\:\:x\cdot_iy=x\prec_i y-y\succ_ix\:\:,\:\:x\bullet_j y=x\circ_j y+y\circ_j x\:\:,\:\:[x,y]_k=x\ast_k y-y\ast_k x),\] 
and similarly in the other case.

\begin{definition} A $\op{Com}_\bullet(\Oh)$-algebra is an $\op{Ass}_\bullet(\Oh)$-algebra with trivial associated $\Oh$-algebra. In other words, the operad $\op{Com}_\bullet(\Oh)$ is generated by non-symmetric products $\star_i$, anti-commutative brackets $\{-,-\}_j$ and commutative products $\circledast_k$, together with the relations obtained from the ones of $\op{Ass}_\bullet(\Oh)$-algebra by further imposing the identities \[x\prec_i y= y\succ_i x=: x\star_i y,\qquad x\circ_j y=-y\circ_j x =:\{x,y\}_j,\qquad x\ast_k y= y\ast_k x=: x\circledast_k y.\]
\end{definition}

\begin{example} We see immediately by Example \ref{ex:ass} that $\op{Com}_\bullet(\sL ie)=\sC om$, the operad of commutative and associative algebras. 
\end{example}

\begin{theorem}\label{th:black} For any operad $\Oh$ in $\mathbf{Op}$, we have natural isomorphisms
$\op{Ass}_\bullet(\Oh)=\sA ss\bullet\Oh$, $\op{Com}_\bullet(\Oh)=\sC om\bullet\Oh$, $\op{preLie}_\bullet(\Oh)=pre\sL ie\bullet\Oh$.
\end{theorem}

The proof is postponed to the next section (it follows from theorems \ref{th:asswhite}, \ref{th:permwhite}, \ref{th:liewhite} and \ref{th:adj}). In the remaining of this section we shall illustrate the result by several explicit computations. 

\begin{remark}\label{rem:curlyeqprec}
	In some of the following computations it is conventient to choose a different basis for the operations of $\op{Ass}_\bullet(\Oh)$ and $\op{preLie}_\bullet(\Oh)$, by replacing the products $\succ_i$ with the opposite products $\curlyeqsucc_i:=-\succ^{op}_i$: with the previous notations $\to_x\cdot'_i\xrightarrow{y}=:\xrightarrow{x\curlyeqsucc_i y}:=\xrightarrow{x}\cdot'_i\,_y\to$.
\end{remark}

\begin{example} Given non-negative integers $p,q,r$, we denote by $\sM ag_{p,q,r}$ the operad generated by $p$ magmatic non-symmetric operations, $q$ magmatic commutative operations and $r$ magmatic anti-commutative operations with no relations among them. The previous arguments show that $\op{Ass}_\bullet(\sM ag_{p,q,r} )=\op{preLie}_\bullet (\sM ag_{p,q,r})=\sM ag_{2p+q+r,0,0}$, whereas $\op{Com}_\bullet(\sM ag_{p,q,r})=\sM ag_{p,r,q}$.
\end{example}

We introduce a further notation.

\begin{definition}\label{def:+x} Given operads $\mathcal{O},\mathcal{P}\in\mathbf{Op}$, we denote by $\mathcal{O}+\mathcal{P}$ the operad defined by saying that an $(\mathcal{O}+\mathcal{P})$-algebra structure on $V$ is the datum of both an $\Oh$-algebra structure and a $\mathcal{P}$-algebra structure on $V$ with no relations among them, while we denote by $\Oh\times\mathcal{P}$ the operad defined in the same way and by further imposing that every triple product involving an operation from $\Oh$ and one from $\mathcal{P}$ vanishes. It is easy to see that the functors $-+-,-\times-:\mathbf{Op}\to\mathbf{Op}$ are Koszul dual in the sense that $(\Oh+\sP)^!=\Oh^!\times\sP^!$ for any pair of operads $\Oh,\sP\in\mathbf{Op}$. \end{definition}

\begin{example} We consider the operad $(\sA ss,\cdot)$ of associative algebras. Given an $\op{Ass}_\bullet(\sA ss)$-algebra structure $(V,\prec,\succ)$ on a vector space, we shall denote the associator of the associated product $\cdot'$ on $C^\ast(\Delta_1;V)$ by $A_{\cdot'}(-,-,-)$. Straightforward computations show that
\[ 0=A_{\cdot'}(_x\to,\to_y,\xrightarrow{z})\:=\: \left(_x\to\cdot'\to_y\right)\cdot'\xrightarrow{z}-\,_x\to\cdot'(\to_y\cdot'\xrightarrow{z}) \:=\: 0+\,_x\to\cdot'\xrightarrow{z\succ y}\:=\:\xrightarrow{x\prec(z\succ y)},\]
\begin{eqnarray} \nonumber 0=A_{\cdot'}\left(\xrightarrow{}_x,{}_y\xrightarrow{},\xrightarrow{z} \right) &=& \xrightarrow{(y\prec z)\succ x}, \\
\nonumber 0=A_{\cdot'}\left({}_x\xrightarrow{},\xrightarrow{y},\xrightarrow{}_z \right) &=& \xrightarrow{(x\prec y)\prec z-x\prec(y\prec z)}, \\
\nonumber 0=A_{\cdot'}\left(\xrightarrow{}_x,\xrightarrow{y},{}_z\xrightarrow{} \right) &=& \xrightarrow{z\succ(y\succ x)-(z\succ y)\succ x}, \\
\nonumber 0=A_{\cdot'}\left(\xrightarrow{x},{}_y\xrightarrow{},\xrightarrow{}_z \right) &=& \xrightarrow{-(y\succ x)\prec z},\\ \nonumber 0=A_{\cdot'}\left(\xrightarrow{x},\xrightarrow{}_y,{}_z\xrightarrow{} \right) &=& \xrightarrow{-z\succ(x\prec y)}. \end{eqnarray}
Conversely, according to Lemma \ref{lem:rel}, the above six relations on $\prec,\succ$ imply the vanishing of $A_{\cdot'}(-,-,-)$ on $C^*(\Delta_1;V)$. Hence, we see that an $\op{Ass}_\bullet( \sA ss)$-algebra structure on $V$ is the datum of two non-symmetric products $\prec,\succ:V^{\ten 2}\to V$ which are associative and such that moreover
\[x\prec(y\succ z)=(x\prec y)\succ z =x\succ(y\prec z)=(x\succ y)\prec z =0,\qquad\forall x,y,z\in V.\]
With the notations from the previous definition we found that that $\op{Ass}_\bullet( \sA ss)=\sA ss\times\sA ss$. By further imposing $x\prec y-y\succ x =: x\cdot y= 0$, that is, $x \prec y = y\succ x=:x\star y$, we see that the operad $\op{Com}_\bullet(\sA ss)$ is generated by a non symmetric product $\star$ such that $(x\star y)\star z=x\star (y\star z)=0$: in other words $\op{Com}_\bullet(\sA ss)= nil\sA ss$, the operad of two step nilpotent associative algebras, in accord with Theorem \ref{th:black} and the computations in \cite{GK}. To compute the relations of $\op{preLie}_\bullet(\sA ss)$ it is convenient to replace the generating set of operations $\prec,\succ$ with the one $\prec,\curlyeqsucc=-\succ^{op}$ as explained in Remark \ref{rem:curlyeqprec}, then we get the following relations according to Lemma \ref{lem:relpl} (notice that $x\cdot y:=x\prec y-y\succ x=x\prec y + x\curlyeqsucc y$)
\begin{eqnarray} \nonumber 0=A_{\cdot'}\left(\xrightarrow{x},\xrightarrow{}_y,\xrightarrow{}_z \right) &=& \xrightarrow{(x\prec y)\prec z-x\prec(y\cdot z)}, \\ 
\nonumber 0=A_{\cdot'}\left(\xrightarrow{}_x,\xrightarrow{y},\xrightarrow{}_z \right) & = & \xrightarrow{(x\curlyeqsucc y)\prec z-x\curlyeqsucc(y\prec z)}, \\ \nonumber 0=A_{\cdot'}\left(\xrightarrow{}_x,\xrightarrow{}_y,\xrightarrow{z} \right) &=& \xrightarrow{(x\cdot y)\curlyeqsucc z - x\curlyeqsucc(y\curlyeqsucc z)}, \end{eqnarray}
which are exactly the dendirform relations for the products $\prec,\curlyeqsucc$, thus $\op{preLie}_\bullet(\sA ss)=\sD end$, the operad of dendriform algebras, which as well known is also $\sD end=pre\sL ie\bullet\sA ss$.
\end{example}
\begin{example} Next we consider the operad $(\sC om,\bullet)$ of commutative and associative algebras. Given an $\op{Ass}_\bullet(\sC om)$ algebra structure $(V,\circ)$ on a vector space, together with the asociated commutative product $\bullet'$ on $C^*(\Delta_1;V)$, similar computations as in the previous example show that the vanishing of the associator $A_{\bullet'}(-,-,-)$ is equivalent to the identities $(x\circ y)\circ z = x\circ(y\circ z)=0$, thus recovering once again $\op{Com}_\bullet(\sA ss)= \sC om\bullet \sA ss = nil\sA ss$. If we further impose that the commutative product $x\bullet y = x\circ y + y\circ x$ on $V$ vanishes, we see that $\op{Com}_\bullet(\sC om)$ is the operad generated by an anti-commutative bracket $\{x,y\}:=x\circ y=-y\circ x$ such that $\{ \{ x,y \}, z \}=0$, or in other words $\op{Com}_\bullet(\sC om)=nil\sL ie$, the operad of two step nilpotent Lie algebras. Finally, a $\op{preLie}_\bullet(\sC om)$-algebra structure on $V$ is the datum of a non-symmetric product $\circ:V^{\ten 2}\to V$ such that, where as usual we denote by $x\bullet y=x\circ y + y\circ x$, \[ 0=A_{\bullet'}(\xrightarrow{x},\to_y,\to_z)=\xrightarrow{x\circ(y\bullet z)-(x\circ y) \circ z},\qquad\forall x,y,z\in V, \] 
and a second relation coming from Lemma \ref{lem:relpl}, namely, $(x\circ y)\circ z= (x\circ z)\circ y$, already follows from this one. We find that $\op{preLie}_\bullet(\sC om)=\sZ inb$, the operad of (right) Zinbiel algebra, according to the well known fact $pre\sL ie\bullet \sC om=\sZ inb$ \cite{V}.
\end{example}

\begin{example} We consider the operad $pre\sL ie$ of (right) pre-Lie algebras. Given an $\op{Ass}_\bullet(pre\sL ie)$-algebra structure $(V,\prec,\succ)$ on $V$, together with the associated dg pre-Lie algebra $(C^\ast(\Delta_1;V),d,\cdot')$, the associator $A_{\cdot'}(-,-,-)$ on $C^\ast(\Delta_1;V)$ is computed as in Example \ref{ex:ass}. Writing the generating relation in $pre\sL ie$ as $0=R(\alpha,\beta,\gamma)=A_{\cdot'}(\alpha,\beta,\gamma)-A_{\cdot'}(\alpha,\gamma,\beta)$, $\forall \alpha,\beta,\gamma\in C^*(\Delta_1;V)$, we find that according to Lemma \ref{lem:rel} this is equivalent to the following relations on $\prec,\succ$ (the remaining three follow from these ones by symmetry of $R$) 
	\[ 0 = R(_x\to,\to_y,\xrightarrow{z})=\xrightarrow{x\prec(z\succ y)-(x\prec z)\prec y+x\prec(z\prec y)}\]
	\[ 0 = R(\to_x,\,_y\to,\xrightarrow{z})=\xrightarrow{(y \prec z)\succ x -y\succ(z\succ x)+(y\succ z)\succ x}\]
	\[ 0 = R(\xrightarrow{x},\,_y\to,\to_z)=\xrightarrow{-(y\succ x)\prec z+ y\succ(x\prec z)}\]
		
These are precisely the relations of dendriform algebra, thus once again $\op{Ass}_\bullet(pre\sL ie)=\sD end$. If we impose that the right pre-Lie product $x \cdot y = x\prec y - y\succ x$ on $V$ vanishes, we get as well known \cite{LV} the Zinbiel relation for $x\star y:=x\prec y=y\succ x$, that is, once again $\op{Com}_\bullet(pre\sL ie)=\sZ inb$. Finally, in a $\op{preLie}_\bullet(pre\sL ie)$-algebra we get the relations (the remaining one follows from the first one and simmetry of $R$)
\[0 = R(\to_x,\to_y,\xrightarrow{z})=\xrightarrow{-z\succ (x\cdot y)-(z\succ y)\succ x+(z\succ x)\prec y-(z\prec y)\succ x}\]
\[0 = R(\xrightarrow{x},\to_y,\to_z)=\xrightarrow{(x\prec y)\prec z -(x\prec z)\prec y- x\prec(y\cdot z -z\cdot y)}\]
The reader will recognize the relations of $L$-dendriform algebra \cite{bai2,V}. We recall the proof of the following fact, giving the two forgetful functors from the category of $(pre\sL ie\bullet pre\sL ie)$-algebras to the one of $pre\sL ie$-algebras. 
\begin{proposition} Let $V$ be a vector space, together with non symmetric products $\prec,\succ:V^{\ten2}\to V$ such that the following relations, where we put $x\cdot y:= x\prec y-y\succ x$, $x\triangleleft y := x\prec y + x\succ y$, $[x,y]:=x\cdot y-y\cdot x=x\triangleleft y-y\triangleleft x$, are verified
	\[ x\prec [y,z]=(x\prec y)\prec z - (x\prec z)\prec y,\qquad (x\succ y)\prec z=x\succ(y\cdot z)+(x\triangleleft z)\succ y,\qquad\forall x,y,z\in V. \]
Then the products $\triangleleft,\cdot$ are right pre-Lie products on $V$ and $[-,-]$ is a Lie bracket.
\end{proposition}
	
\begin{proof} The relations in the claim of the lemma imply the pre-Lie relation for $\cdot$ according to the previous computations and the proof of Lemma \ref{lem:rel}. The fact that $[-,-]$ is a Lie bracket follows. Finally, we may write the second relation as $A_{\succ}(x,z,y)=(x\succ y)\prec z - (x\prec z)\succ y -x\succ(y\prec z)$: substituting in 
		\[ A_{\triangleleft}(x,y,z)=A_{\prec}(x,y,z)+A_{\succ}(x,y,z)+(x\prec y)\succ z-x\prec(y\succ z)+(x\succ y)\prec z-x\succ(y\prec z)\]
		we find that 
		\begin{multline*}
		A_{\triangleleft}(x,y,z)= A_\prec(x,y,z)-x\prec(y\succ z)+\\+(x\succ z)\prec y-(x\prec y)\succ z-x\succ(z\prec y)+(x\prec y)\succ z+(x\succ y)\prec z-x\succ(y\prec z).
		\end{multline*}
		The bottom row is clearly symmetric in $y$ and $z$, while the same statement for the right hand side of the top row is just another way of writing the first relation in the claim of the lemma.\end{proof}
\end{example}

\begin{example}\label{ex:leib} An interesting example is the operad $\sL eib$ of (right) Leibniz algebras: this is the operad generated by a single non symmetric product $\cdot$ and the relation $0=R_\cdot(x,y,z):=(x\cdot y)\cdot z-x\cdot(y\cdot z)-(x\cdot z)\cdot y$. From Lemma \ref{lem:rel} we get the following five indipendent relations on $\prec,\succ$
	\begin{multline*} 0=R_{\cdot'}(_x\to,\to_y,\xrightarrow{z})\:=\: \left(_x\to\cdot'\to_y\right)\cdot'\xrightarrow{z}-\,_x\to\cdot'(\to_y\cdot'\xrightarrow{z})-(_x\to\cdot'\xrightarrow{z})\cdot'\to_y \:= \\ 0+\,_x\to\cdot'\xrightarrow{z\succ y}-\xrightarrow{x\prec z}\cdot'\to_y\:=\:\xrightarrow{x\prec(z\succ y)-(x\prec z)\prec y},\end{multline*}
	
	\begin{eqnarray} \nonumber 0=R_{\cdot'}\left(\xrightarrow{}_x,{}_y\xrightarrow{},\xrightarrow{z} \right) &=& \xrightarrow{(y\prec z)\succ x-y\succ(z\succ x)}, \\
	\nonumber 0=R_{\cdot'}\left({}_x\xrightarrow{},\xrightarrow{y},\xrightarrow{}_z \right) &=& \xrightarrow{(x\prec y)\prec z-x\prec(y\prec z)}, \\
	\nonumber 0=R_{\cdot'}\left(\xrightarrow{}_x,\xrightarrow{y},{}_z\xrightarrow{} \right) &=& \xrightarrow{z\succ(y\succ x)-(z\succ y)\succ x}, \\
	\nonumber 0=R_{\cdot'}\left(\xrightarrow{x},{}_y\xrightarrow{},\xrightarrow{}_z \right) &=& \xrightarrow{-(y\succ x)\prec z+y\succ(x\prec z)},
	\end{eqnarray}
	and the remaining one is equivalent to the last one.
These are the precisely the relations of diassociative algebra on $(V,\prec,\succ)$, in other words we found that $\op{Ass}_\bullet(\sL eib)=di\sA ss$, the operad of diassociative algebras. If we further impose $0=x\cdot y=x\prec y -y\succ x$ and we put as usual $x\star y :=x\prec y=y\succ x$, the previous relations reduce to the two independent ones $(x\star y)\star z=x\star(y\star z)$, $(x\star y)\star z=x\star(z\star y)$: in other words, we found that $\op{Com}_\bullet(\sL eib)=\sP erm$, the operad of (right) permutative algebras. Finally,  we consider the operad $\op{preLie}_\bullet(\sL eib)$: this is generated by non-symmetric products $\prec,\succ$ and the relations, where as usual $x\cdot y= x\prec y- y\succ x$, 
\begin{eqnarray} \nonumber 0=R_{\cdot'}\left(\xrightarrow{x},\xrightarrow{}_y,\xrightarrow{}_z \right) &=& \xrightarrow{(x\prec y)\prec z-(x\prec z)\prec y-x\prec(y\cdot z)}, \\
\nonumber 0=R_{\cdot'}\left(\xrightarrow{}_x,\xrightarrow{y},\xrightarrow{}_z \right) &=& \xrightarrow{-(y\succ x)\prec z+(y\prec z)\succ x+y\succ(x\cdot z)}, \\
\nonumber 0=R_{\cdot'}\left(\xrightarrow{}_x,\xrightarrow{}_y,\xrightarrow{z} \right) &=& \xrightarrow{-z\succ(x\cdot y)-(z\succ y)\succ x+(z\succ x)\prec y}.
\end{eqnarray}
These are equivalent to the following relations for $\prec,\succ$
\begin{multline*} x\prec(y\cdot z)=(x\prec y)\prec z -(x\prec z)\prec y,\qquad (x\succ y)\succ z=(x\prec y)\succ z, \\ x\succ(y\cdot z)=(x\succ y)\prec z-(x\prec z)\succ y. \end{multline*}

\end{example}

\begin{example}\label{ex:pois} We consider the operad $(\sP ois,\bullet, [-,-])$ of commutative Poisson algebras. By the previous computations, an $\operatorname{Ass}_\bullet(\sP ois)$-algebra structure on $V$ is the datum $(V,\circ,\ast)$ of an $\op{Ass}_\bullet(\sC om)=nil\sA ss$-algebra structure $(V,\circ)$ and an $\op{Ass}_\bullet(\sL ie)=\sA ss$-algebra structure $(V,\ast)$ satisfying the additional relations induced, as in Lemma \ref{lem:rel}, by the Poisson identity $0=[x,y\bullet z]-[x,y]\bullet z - y\bullet[x,z]$ in the operad $\sP ois$. These additional relations are easily computed, for instance
	\begin{multline*}  0=\left[\xrightarrow{x},\,_y\to\bullet'\to_z\right]' - \left[\xrightarrow{x},\,_y\to\right]'\bullet'\to_z-\,_y\to\bullet'\left[\xrightarrow{x},\to_z\right]' = \\=\,0\,-\xrightarrow{-y\ast x}\cdot'\to_z -\,_y\to\cdot'\xrightarrow{x\ast z}\:=\:\xrightarrow{(y\ast x)\circ z - y\circ (x\ast z)}.
	\end{multline*} 
	Proceeding in this way we find the following relations for $\circ,\ast$
	\[ (x\ast y)\circ z= x\ast(y\circ z)=x\circ(y\ast z)= (x\circ y)\ast z.\] 
	Similarly, a $\operatorname{Com}_\bullet(\sP ois)$-algebra structure on $V$ is the datum $(V,\{-,-\},\circledast)$ of a $\operatorname{Com}_\bullet(\sC  om)=nil\sL ie$-algebra structure $(V,\{-,-\})$ and a $\operatorname{Com}_\bullet(\sL ie)=\sC om$-algebra structure $(V,\circledast)$ such that moreover the above identities hold. Consider the one $\{x,y\}\circledast z=\{x\circledast y,z \} $: as the left and right hand side are respectively symmetric and anti-symmetric in $x$ and $y$ both vanish, thus we have the relations
	\[ \{x\circledast y,z \}=x\circledast\{y,z\}=0,\qquad\forall x,y,z\in V.\]
	In other words, we found $\operatorname{Com}_\bullet(\sP ois)=nil\sL ie\times\sC om$. Finally, a $\op{preLie}_\bullet(\sP ois)$-algebra structure on $V$ is the datum $(V,\circ,\ast)$ of a $\op{preLie}_\bullet(\sC om)=\sZ inb$-algebra structure $(V,\circ)$ and a $\op{preLie}_\bullet(\sL ie)=pre\sL ie$-algebra structure $(V,\ast)$, such that moreover, where as usual we denote by $x\bullet y =x\circ y+ y\circ x$ and $[x,y]=x\ast y-y\ast x$ the associated $\sC om$- and $\sL ie$-algebra structures on $V$ respectively,
	\begin{multline*}  0=\left[\xrightarrow{x},\to_y\bullet'\to_z\right]' - \left[\xrightarrow{x},\,\to_y\right]'\bullet'\to_z-\to_y\bullet'\left[\xrightarrow{x},\to_z\right]' = \\=\left[\xrightarrow{x},\to_{y\bullet z}\right]'-\xrightarrow{x\ast y}\bullet'\to_z -\to_y\bullet'\xrightarrow{x\ast z}=\xrightarrow{x\ast(y\bullet z) -(x\ast y)\circ z - (x\ast z)\circ y},
	\end{multline*} 
	\begin{multline*}  0=\left[\to_x,\xrightarrow{y}\bullet'\to_z\right]' - \left[\to_x,\xrightarrow{y}\right]'\bullet'\to_z-\xrightarrow{y}\bullet'\large[\to_x,\to_z\large]' = \\=\left[\to_x,\xrightarrow{y\circ z}\right]'-\xrightarrow{-y\ast x}\bullet'\to_z -\xrightarrow{y}\bullet'\to_{[x,z]}=\xrightarrow{-(y\circ z)\ast x+(y\ast x)\circ z - y\circ [x,z]}.
	\end{multline*} 
	We conclude that $\op{preLie}_\bullet(\sP ois)=pre\sP ois$, Aguiar's operad of right pre-Poisson algebras \cite{A}, in accord with the fact \cite{U} that $pre\sL ie\bullet\sP ois = pre\sP ois$.
\end{example}

\begin{example}\label{ex:perm} We consider the operad $\sP erm$: recall that a $\sP erm$-algebra structure, or (right) permutative algebra structure, on $V$ is the datum of an associative product $\cdot$ which is commutative on the right hand side whenever there are three or more variables, or, in other words, satisfies the relations $(x\cdot y)\cdot z=x\cdot(y\cdot z)=x\cdot(z\cdot y)$. Thus, an $\op{Ass}_\bullet(\sP erm)$-algebra structure on $V$ is the datum of an $\op{Ass}_\bullet(\sA ss)=\sA ss\times\sA ss$ algebra structure $(V,\prec,\prec)$, as in Example \ref{ex:ass}, such that moreover
\[ 0 \:=\: _x\to\cdot'\left( \to_y\cdot'\xrightarrow{z}\right)\,-\,_x\to\cdot'\left( \xrightarrow{z}\cdot'\to_y\right) =\xrightarrow{-x\prec(z\prec y)-x\prec(z\prec y)}=\xrightarrow{-x\prec(y\prec z)}\]
\[ 0 \:=\: \to_x\cdot'\left( _y\to\cdot'\xrightarrow{z}\right)-\to_x\cdot'\left( \xrightarrow{z}\cdot'\,_y\to \right) =\xrightarrow{-(y\prec z)\prec x-(y\prec z)\prec x}=\xrightarrow{-(y\prec z)\prec x}.\]
In other words, $\op{Ass}_\bullet(\sP erm)=nil\sA ss\times nil\sA ss$. Further imposing $0=x\prec y-y\prec x=x\cdot y$, we find that $\op{Com}_\bullet(\sP erm)=nil\sA ss$. Finally, as in Remark \ref{rem:curlyeqprec} and Example \ref{ex:ass}, it is convenient in the compuation of $\op{preLie}_\bullet(\sP erm)$ to replace the product $\prec$ by the opposite one $\curlyeqsucc:=-\prec^{op}$, satisfying the formulas $\to_x\cdot'\xrightarrow{y}=\xrightarrow{x}\cdot'\,_y\to=:\xrightarrow{x\curlyeqsucc y}$. In fact, as we saw in \ref{ex:ass} the relations of $\op{preLie}_\bullet(\sP erm)$-algebra translate in the relations of $\op{preLie}_\bullet(\sA ss)=\sD end$-algebra for the products $\prec,\curlyeqsucc$, and we get the additional relations (where as usual $x\cdot y=x\prec y-y\prec x=x\prec y+ x\curlyeqsucc y$)
\[ 0= \xrightarrow{x}\cdot'\left( \to_y\cdot'\to_z\right)-\xrightarrow{x}\cdot'\left(\to_z\cdot'\to_y\right)=\xrightarrow{x\prec(y\cdot z) - x\prec(z\cdot y)}\]
\[ 0= \to_x\cdot'(\xrightarrow{y}\cdot'\to_z)-\to_x\cdot'\left( \to_z\cdot'\xrightarrow{y}\right)= \xrightarrow{x\curlyeqsucc(y\prec z)-x\curlyeqsucc(z\curlyeqsucc y)}.\]
This is in accord with \cite[Theorem 21]{V}.
\end{example}

\begin{definition} Given an operad $(\Oh,\cdot_i,\bullet_j,[-,-]_k)$, an $\Oh_{adm}$-algebra structure on $V$ is the datum of non-symmetric operations $\prec_i,\succ_i,\circ_j.\ast_k:V^{\ten 2}\to V$ such that the induced operations $x\cdot_i y:=x\prec_i y+x\succ_iy$, $x\bullet_j y= x\circ_j y+y\circ_j x$, $[x,y]_k=x\ast_k y-y\ast_k x$ define an $\Oh$-algebra structure on $(V,\cdot_i,\bullet_j,[-,-]_k)$. This defines the operad $(\Oh_{adm},\prec_i,\succ_i,\circ_j,\ast_k)$ of $\Oh$ admissible algebras.
\end{definition}

\begin{example}\label{ex:lieadm} We consider the operad $\sL ie_{adm}$ of Lie admissible algebras, that is, the operad generated by a non-symmetric operation $\cdot$ such that the commutator is a Lie bracket. We shall write the only relation in $\sL ie_{adm}$ as $0=R_\cdot(x_1,x_2,x_3)=\sum_{\sigma\in S_3}\varepsilon(\sigma)A_\cdot(x_{\sigma(1)},x_{\sigma(2)},x_{\sigma(3)})$, where $\varepsilon(\sigma)$ is the sign of the permutation $\sigma$ and as usual $A_\cdot(-,-,-)$ is the asociator of $\cdot$. Given an $\operatorname{Ass}_\bullet(\sL ie_{adm})$-algebra structure $(V,\prec,\succ)$ on $V$, together with the associated $\sL ie_{adm}$-algebra structure on $C^*(\Delta_1;V)$, the associator $A_{\cdot'}(-,-,-)$ is computed as in Example \ref{ex:ass}. By Lemma \ref{lem:rel} the vanishing of $R_\cdot'(-,-,-)$ is equivalent to the following relation on $\prec,\succ$ (the others follow from this one by symmetry of $R$) 
\begin{multline*} 0= R_{\cdot'}(_x\to,\to_y,\xrightarrow{z})=A_{\cdot'}(_x\to,\to_y,\xrightarrow{z})+ A_{\cdot'}(\to_y,\xrightarrow{z},\,_x\to)+A_{\cdot'}(\xrightarrow{z},\,_x\to,\to_y)-\\-A_{\cdot'}(\to_y,\,_x\to,\xrightarrow{z})-A_{\cdot'}(_x\to,\xrightarrow{z},\to_y)-A_{\cdot'}(\xrightarrow{z},\to_y,\,_x\to)=
\end{multline*} 
\[=\xrightarrow{x\prec(z\succ y) + x\succ(z\succ y)-(x\succ z)\succ y - (x\succ z)\prec y-(x\prec z)\succ y-(x\prec z)\prec y -x\prec(z\prec y)+x\succ(z\prec y)}\]
This is exactly the relation saying that the product $x\cup y:=x\prec y + x\succ y$ on $V$ (so denoted to distinguish it from the Lie admissible product $x\cdot y=x\prec y-y\succ x$) is associative. Hence, we found $\op{Ass}_\bullet(\sL ie_{adm})=\sA ss_{adm}$. Further imposing $x\prec y=y\succ x=:x\star y$, we see that $x\cup y= x\star y + y\star x$ is an associative and commutative product, and thus that $\op{Com}_\bullet(\sL ie_{adm})=\sC om_{adm}$. For the operad $(\operatorname{preLie}_\bullet(\sL ie_{adm}),\prec,\succ)$ we find the relation (and the others follow by symmetry of $R$), where as usual $x\cdot y=x\prec y-y\succ x$,
\begin{multline*} 0= R_{\cdot'}(\to_x,\to_y,\xrightarrow{z})=A_{\cdot'}(\to_x,\to_y,\xrightarrow{z})+ A_{\cdot'}(\to_y,\xrightarrow{z},\to_x)+A_{\cdot'}(\xrightarrow{z},\to_x,\to_y)-\\-A_{\cdot'}(\to_y,\to_x,\xrightarrow{z})-A_{\cdot'}(\to_x,\xrightarrow{z},\to_y)-A_{\cdot'}(\xrightarrow{z},\to_y,\to_x)= 
\end{multline*} 
\[=\xrightarrow{-z\succ(x\cdot y) - (z\succ y)\succ x - (z\succ y)\prec x + (z\prec x)\succ y + (z\prec x)\prec y-z\prec(x\cdot y)+z\succ (y\cdot x) + (z\succ x)\succ y + (z\succ x)\prec y -(z\prec y)\succ x -(z\prec y)\prec x+ z\prec(y\cdot x)}\]
Putting $x\cup y = x\prec y + x\succ y$, the reader will check that the above relation can be rewritten as $0=A_\cup(z,x,y)-A_\cup(z,y,x)$, and thus we found $\op{preLie}_\bullet(\sL ie_{adm})=pre\sL ie_{adm}$.\end{example} 

\begin{remark}
	The latter example suggests $\Oh\bullet\sL ie_{adm}=\Oh_{adm}$ for every operad $\Oh$ in $\mathbf{Op}$.	
\end{remark}

\begin{example}\label{ex:postlie} We consider the operad $(post\sL ie,\cdot,[-,-])$ of post-Lie algebras: this is the operad generated by a non-symmetric product $\cdot$ and a Lie bracket $[-,-]$ satisfying the relations $0=R_1(x,y,z)=[x,y]\cdot z-[x,y\cdot z]-[x\cdot z,y]$ and $0=R_2(x,y,z)=x\cdot(y\cdot z +z\cdot y + [z,y])-(x\cdot y)\cdot z+(x\cdot z)\cdot y$. In the operad $(\op{Ass}_\bullet(\Oh),\prec,\succ,\ast)$ we get the associativity relation for the product $\ast$ and the six independent relations 
\[ 0=R_1(_x\to,\xrightarrow{y},\to_z)=\xrightarrow{(x\ast y)\prec z-x\ast(y\prec z)}, \]
\[ 0=R_1(\to_x,\xrightarrow{y},_z\to)=\xrightarrow{z\succ(y\ast x)-(z\succ y)\ast x},\]
\[ 0=R_1(_x\to,\to_y,\xrightarrow{z})=\xrightarrow{x\ast(z\succ y)-(x\prec z)\ast y},\]
\[ 0=R_2(_x\to,\xrightarrow{y},\to_z)=\xrightarrow{x\prec(y\prec z+z\prec y+y\ast z)-(x\prec y)\prec z},\]
\[ 0=R_2(\to_x,\xrightarrow{y},_z\to)= \xrightarrow{(z\succ y+ y\prec z +z\star y)\succ x-z\succ(y\succ x)},\]
\[ 0=R_2( \xrightarrow{x},_y\to,\to_z)=\xrightarrow{(y\succ x)\prec z-y\succ(x\prec z)}.\]
These are the relations of dendriform trialgebra \cite{LR}, hence we found $\op{Ass}_\bullet(post\sL ie)=\mathcal{T}ridend$. If we further impose $x\prec y= y\succ x=: y\star x$, $x\star y=y\star x=:x\circledast y$, we find that the operad $(\op{Com}_\bullet(post\sL ie),\star,\circledast)$ is defined by the following relations
\[ (x\circledast y)\circledast z=x\circledast(y\circledast z),\quad(x\circledast y)\star z=x\circledast(y\star z),\quad(x\star y)\star z =x\star(y\star z+z\star y +y\circledast z). \]
We shall denote this operad by $post\sC om:=\op{Com}_\bullet(post\sL ie)$. We leave to the interested reader the computation of $\op{preLie}_\bullet(post\sL ie)$. We have morphism of operads  
\[ (\sL ie,\{-,-\})\to(post\sL ie,\cdot,[-,-]):\{-,-\}\:\:\longrightarrow\:\:\cdot\:-\:\cdot^{op}\:+\:[-,-],\]
\[ (\sA ss,\cup)\to(\mathcal{T}ridend,\prec,\succ,\ast):\cup\:\:\longrightarrow\:\:\prec+\succ+\ast,\]
\[(\sC om,\bullet)\to(post\sC om,\star,\circledast):\bullet\:\:\longrightarrow\:\: \star+\star^{op}+\circledast, \] 
in the first (as well as the second) case this is well known, the other two follow by functoriality. 
\end{example}

\begin{example} We consider the operad generated by a Lie bracket $[-,-]$ and a right Leibniz product $\cdot$ satisfying the additional relations 
	\[ [x,y]\cdot z=[x,y\cdot z]+[x\cdot z,y],\qquad x\cdot[y,z]=x\cdot(y\cdot z). \]
	It can be checked that this is the Koszul dual $(post\sC om^!, \cdot,[-,-])$ of the operad $post\sC om:=\op{Com}_\bullet(post\sL ie)$ from the previous example. To aid in the computations, we notice that this is the same, up to changing the sign of the operation $\cdot$, as the operad we obtain from $(post\sL ie,\cdot,[-,-])$ by further imposing the right Leibniz identity for $\cdot$. By the previous example and Example \ref{ex:leib}, we see that the operad $\op{Ass}_\bullet(post\sC om^!)$ is the same as the one we obtain from the operad $\mathcal{T}ridend$ of dendriform trialgebras by further imposing the diassociativity relations for $\prec,\succ$, and then changing the signs of $\prec,\succ$: the reader will readily verify that this is the operad $\op{Ass}_\bullet(post\sC om^!)=\mathcal{T}riass$ of triassociative algebras by Loday and Ronco \cite{LR}. Likewise, $\op{Com}_\bullet(post\sC om^!)$ is the operad we obtain from $(post\sC om,\star,\circledast)$, defined as in the previous example, by further imposing that $\star$ is a right permutative product, and then changing the sign of $\star$: the reader will readily verify that this is the same as the operad $\sC omtrias$ of commutative trialgebras \cite{Vpos}.   
\end{example}
\newcommand{\sR}{\mathcal{R}}

\section{A Koszul dual trick to compute Manin white products}

We denote by $-\circ-:\mathbf{Op}\times\mathbf{Op}\to\mathbf{Op}$ the Manin white product of operads \cite{V}, it gives $\mathbf{Op}$ a structure of symmetric monoidal category, and by $-^!:\mathbf{Op}\to\mathbf{Op}$ the Koszul duality functor. Recall that the Manin white and black products are Koszul dual, in the sense that $(\Oh\bullet\sP)^!=\Oh^!\circ\sP^!$ for every pair of operads $\Oh$ and $\sP$ in $\mathbf{Op}$, and moreover for every operad $\Oh$ the functors $\xymatrix{\Oh\bullet-:\mathbf{Op}\ar@<2pt>[r]& \mathbf{Op}:\Oh^!\circ-\ar@<2pt>[l]}$ form an adjoint pair. recall that $\sA ss^!=\sA ss$, $\sC om^!=\sL ie$ and $pre\sL ie^!=\sP erm$. In the previous section we introduced easily computable functors $\op{Ass}_\bullet(-)$, $\op{Com}_\bullet(-),\op{preLie}_\bullet(-):\mathbf{Op}\to\mathbf{Op}$ and claimed that they coincide with the respective Manin black products $\sA ss\bullet-,\sC om\bullet-, pre\sL ie\bullet-$. The aim of this section is to introduce the respective right adjoint functors $\op{Ass}_\circ(-),\op{Lie}_\circ(-),\op{Perm}_\circ(-):\mathbf{Op}\to\mathbf{Op}$ (the adjointness relation will be shown at the very end of the section) and prove that they in fact coincide with the Manin white products $\sA ss\circ-,\sL ie\circ-,\sP erm\circ-$: 
this will also complete the proof of Theorem \ref{th:black}.

We consider first the case of $\operatorname{Ass}_\circ(-)$. Given an (as usual, binary, quadratic and finitely generated) operad $(\Oh,\cdot_i,\bullet_j,[-,-]_k)$, the operad $\op{Ass}_\circ(\Oh)$ is generated by non symmetric operations $\prec_i,\succ_i,\circ_j,\ast_k$.  We want a morphism of operads $\mu_{\Oh}:\op{Ass}_\bullet(\op{Ass}_\circ(\Oh))\to\Oh$ corresponding to the counit of the adjunction  $\xymatrix{\op{Ass}_\bullet(-):\mathbf{Op}\ar@<2pt>[r]& \mathbf{Op}:\op{Ass}_\circ(-)\ar@<2pt>[l]}$: in other words, by definition of $\op{Ass}_\bullet(-)$, given an $\Oh$-algebra structure $(V,\cdot_i,\bullet_j,[-,-]_k)$ on a vector space $V$ we want an induced local dg $\op{Ass}_\circ(\Oh)$-algebra structure on the complex $C^*(\Delta_1;V)=C^*(\Delta_1;\mathbb{K})\otimes V$. Denoting by $\cup$ the usual cup product of cochains on $C^*(\Delta_1;\mathbb{K})$, the $\op{Ass}_\circ(\Oh)$-algebra structure on $C^*(\Delta_1;V)$ will be given by the tensor product operations $\prec_i=\cup\otimes\cdot_i$, $\succ_i=\cup\otimes\cdot_i^{op}$, $\circ_j=\cup\ten\bullet_j$, $\ast_k=\cup\ten[-,-]_k$: explicitly, with the same notations as in the previous section, 
\[ \to_x\prec_i\to_y\:=\:\to_{x\cdot_i y}\:=\:\to_y\succ_i\to_x,\qquad \to_x\circ_j\to_y\:=\:\to_{x\bullet_j y},\qquad\to_x\ast_k\to_y\:=\: \to_{[x,y]_k}, \] \[ _x\to\prec_i\,_y\to\:=\:_{x\cdot_i y}\to\:=\:_y\to\succ_i\,_x\to,\qquad _x\to\circ_j\,_y\to\:=\:_{x\bullet_j y}\to,\qquad_x\to\ast_k\,_y\to\:=\: _{[x,y]_k}\to, \] \[ _x\to\prec_i\xrightarrow{y}=_y\to\succ_i\xrightarrow{x}\:=\:\xrightarrow{x\cdot_i y}\:=\:\xrightarrow{x}\prec_i\to_y= \xrightarrow{y}\succ_i\to_{x},\] \[ _x\to\circ_j\xrightarrow{y}\:=\:\xrightarrow{x\bullet_j y}\:=\:\xrightarrow{x}\circ_j\to_y,\qquad  _x\to\ast_k\xrightarrow{y}\:=\:\xrightarrow{[x,y]_k}\:=\:\xrightarrow{x}\ast_k\to_y, \] 
and the remaining products vanish. It is immediately seen that these operations satisfy the Leibniz identity with respect to $d$ and the locality assumption from Definition \ref{def:ManinBlack}. We will generically call the operations $\prec_i,\succ_i,\circ_j,\ast_k$ on $C^*(\Delta_1;V)$ the cup products. 

\begin{definition} Given an operad $(\Oh,\cdot_i,\bullet_j,[-,-]_k)$ as usual, the operad $\op{Ass}_\circ(\Oh)$ is generated by non-symmetric operations $\prec_i,\succ_i,\circ_j,\ast_k$, together with the larger set of relations making $C^*(\Delta_1;V)$ with the cup products a local dg $\op{Ass}_\circ(\Oh)$-algebra for every $\Oh$-algebra $V$.\end{definition}

Next we want to describe a generating set of relations of $\op{Ass}_\circ(\Oh)$. Given a terniary operation $R(-,-,-):C^*(\Delta_1;V)^{\ten3}\to C^*(\Delta_1;V)$ in the cup products, we write $R=\sum_{\sigma\in S_3}R_\sigma$, where $S_3$ is the symmetric group and $R_\sigma$ is spanned by triple products permuting the variables according to $\sigma$. As in the proof of Lemma \ref{lem:rel}, we see that the relation $R(-,-,-)=0$ holds in $C^*(\Delta_1;V)$ if and only if
\begin{multline*}0=R(\xrightarrow{x},\to_y,\,_z\to)=R(\xrightarrow{x},\,_y\to,\to_z)=R(\to_x,\xrightarrow{y},\,_z\to)=\\=R(_x\to,\xrightarrow{y},\to_z)=R(\to_x,\,_y\to,\xrightarrow{z})=R(_x\to,\to_y,\xrightarrow{z})=0\end{multline*}
for all $x,y,z\in V$. We claim that this is true if and only if $R_\sigma(-,-,-)=0$ holds in $C^*(\Delta_1;V)$ for every $\sigma\in S_3$, where the \vr if'' implication is obvious. To fix the ideas, we consider the identical permutation $\id\in S_3$ and prove $R(-,-,-)=0\solose R_{\id}(-,-,-)=0$: the remaining implications are proved similarly. We obtain the desired conclusion again by Lemma \ref{lem:rel}, the point is that the only non-vanishing triple products that we can form out of the cochains $_1\to,\to_1,\xrightarrow{1}\in C^*(\Delta_1;\mathbb{K})$ are $(_1\to\cup\xrightarrow{1})\cup\to_1\:=\:\xrightarrow{1}\:=\:_1\to\cup(\xrightarrow{1}\cup\to_1)$: this shows at once that $R_{\id}(\xrightarrow{x},\to_y,\,_z\to)=R_{\id}(\xrightarrow{x},\,_y\to,\to_z)=R_{\id}(\to_x,\xrightarrow{y},\,_z\to)=R_{\id}(\to_x,\,_y\to,\xrightarrow{z})=R_{\id}(_x\to,\to_y,\xrightarrow{z})=0$, whereas $R_{\id}(_x\to,\xrightarrow{y},\to_z)=R(_x\to,\xrightarrow{y},\to_z)=\xrightarrow{R'(x,y,z)}=0$ by hypothesis, where $R'(-,-,-):V^{\ten3}\to V$ is a terniary operation in $\Oh$; finally, if this holds for a generic $\Oh$-algebra $V$, we conclude that $R'(-,-,-)=0$ is a relation in the operad $\Oh$. We sum up the previous discussion in the following lemma.
\begin{lemma}\label{lem:relwhite} The operad $\op{Ass}_\circ(\Oh)$ has a generating set of relations $R(-,-,-)=0$ which are non-symmetric (where a relation is non-symmetric if $R=R_{\id}$), consisting of all the possible splittings, as in the following 
	\[ R(_x\to,\xrightarrow{y},\to_z)=\xrightarrow{R'(x,y,z)}=0,\]
of a relation $R'(-,-,-)=0$ in $\Oh$ via the cup products on $C^*(\Delta_1;V)$ (where $V$ is a generic $\Oh$-algebra).\end{lemma}
It is best to illustrate the previous lemma by some examples.
\begin{example} We consider the operad $(\sC om,\bullet)$: in this case $R'$ as in the previous lemma has to be non-symmetric itself, and since the space of non-symmetric relators of $\sC om$ is one-dimensional, spanned by the associativity relation, we only get the splitting
	\[ (_x\to \circ \xrightarrow{y})\circ\to_y - _x\to\circ(\xrightarrow{y}\circ\to_z) =\xrightarrow{(x\bullet y)\bullet z-x\bullet (y\bullet z)}=0,\]
	telling us that the cup product $\circ$ is associative: thus $\op{Ass}_\circ(\sC om)=\sA ss$ as expected. In the case of the operad $(\sL ie,[-,-])$, again $R'$ as in the previous Lemma has to be non-symmetric, and since the space of relators of the operad $\sL ie$ is one-dimensional generated by the Jacobi identity, and this can't be written in a non-symmetric form, we find $\op{Ass}_\circ(\sL ie)=\sM ag_{1,0,0}$, in accord with our computation of the Koszul dual $\op{Ass}_\bullet(\sC om)=nil\sA ss$. 
	\end{example}

\begin{remark} Whereas the computation of the functor $\op{Ass}_\bullet(-)$ is completely mechanical, as we apply Lemma \ref{lem:rel} to a generating set of relations of $\Oh$ to get a generating set of relations of $\op{Ass}_\bullet(\Oh)$, the previous lemma is not as effective in the compution of $\op{Ass}_\circ(\Oh)$: we illustrate this fact by considering the operad $(\sP ois,\bullet,[-,-])$. By the previous example, the operad $\op{Ass}_\circ(\sP ois)$ is generated by an associative product $\circ$ and a magmatic product $\ast$. Since there is no way to split the Poisson identity $[x\bullet y, z] -x\bullet[y,z] - [x,z]\bullet y=0$ as in the previous lemma, as this can't be written in a non-symmetric form, we may be tempted to conclude that $\op{Ass}_\circ(\sP ois)=\sA ss+\sM ag_{1,0,0}$, but this is in contrast with our computation of the Koszul dual $\op{Ass}_\bullet(\sP ois)$ in Example \ref{ex:pois}. In fact, looking closely we find the splitting
\begin{multline*} (_x\to\circ \xrightarrow{y})\ast\to_z-_x\to\circ(\xrightarrow{y}\ast\to_z) + (_x\to\ast\xrightarrow{y})\circ\to_z-_x\to\ast(\xrightarrow{y}\circ\to_z) = \\ = \xrightarrow{[x\bullet y,z ]-x\bullet[y,z]+[x,y]\bullet z-[x,y\bullet z]}=\xrightarrow{[x,z]\bullet y - y\bullet[x,z]}=0.
\end{multline*}
So we find the relation $(\alpha\circ\beta)\ast\gamma-\alpha\circ(\beta\ast\gamma)+(\alpha\ast\beta)\circ\gamma-\alpha\ast(\beta\circ\gamma)=0$, $\forall \alpha,\beta,\gamma\in C^*(\Delta_1;V)$, in the operad $\op{Ass}_\circ(\sP ois)$. Together with the associativity relation for $\circ$, these generate the relations of $\op{Ass}_\circ(\sP ois)$: the easiest way to see this is to check that the operad $\op{Ass}_\circ(\sP ois)$ defined in this way is in fact the Koszul dual of $\op{Ass}_\bullet(\sP ois)$ from \ref{ex:pois}. Returning to the initial remark, the difficulty in applying Lemma \ref{lem:relwhite} is that we have to look for the splittings of any relation $R'$ in $\Oh$, and we can't limit ourselves to let $R'$ vary in a generating set of these relations, so there is no telling us in general if we are missing some of the possible splittings. As in the previous computation, the easiest way around this, and what we will do in practice in the following examples, is to take advantage of the computation of the Koszul dual $\op{Ass}_\bullet(\Oh^!)$ from the previous section. Of course, this depends on the yet to be given proofs of theorems \ref{th:black}, \ref{th:asswhite},\ref{th:permwhite} and \ref{th:liewhite}.
\end{remark}
\begin{example} We consider the operad $(\sA ss,\cdot)$: in this case we know from Example \ref{ex:ass} that $\op{Ass}_\circ(\sA ss)=(\sA ss\bullet\sA ss)^!=(\sA ss\times\sA ss)^!=\sA ss+\sA ss$ (with the functors $-\times-$ and $-+-$ as in Definition \ref{def:+x}), and in fact we find the relations
	\[ (_x\to\prec\xrightarrow{y})\prec\to_z-_x\to\prec(\xrightarrow{y}\prec \to_z)=\xrightarrow{(x\cdot y)\cdot z-x\cdot(y\cdot z)}=0,  \]
	\[ (_x\to\succ\xrightarrow{y})\succ\to_z-_x\to\succ(\xrightarrow{y}\succ \to_z)=\xrightarrow{z\cdot(y\cdot x)-(z\cdot y)\cdot x}=0,  \] 
	in the operad $(\op{Ass}_\circ(\sA ss),\prec,\succ)$.
	
	We consider the operad $(pre\sL ie, \cdot)$. The right pre-Lie relation $(x\cdot y)\cdot z-x\cdot(y\cdot z)-(x\cdot z)\cdot y+x\cdot(z\cdot y)=0$ can't be splitted as in Lemma \ref{lem:relwhite}, for instance because there is no argument remaining inside every parenthesis. As the previous remark illustrates, this alone wouldn't be enough to conclude $\op{Ass}_\circ(pre\sL ie)=\sM ag_{2,0,0}$: on the other hand, this is true since it agrees with the computation of the Koszul dual $\op{Ass}_\bullet(\sP erm)=nil\sA ss\times nil\sA ss$ in Example \ref{ex:perm}.
	
	We consider the operad $(\sL eib,\cdot)$. Again, the right Leibniz relation  $(x\cdot y)\cdot z-x\cdot(y\cdot z)-(x\cdot z)\cdot y=0$ can't be splitted as in Lemma \ref{lem:relwhite}, on the other hand, this imply the relation $x\cdot(y\cdot z+z\cdot y) = (x\cdot y)\cdot z-(x\cdot z)\cdot y +(x\cdot z)\cdot y-(x\cdot y)\cdot z=0$ in $\sL eib$: for the latter, we find the splittings
	\[ _x\to\prec (\xrightarrow{y}\prec\to_z + \xrightarrow{y}\succ\to_z ) =\xrightarrow{x\cdot(y\cdot z+z\cdot y)}=0,\]
	\[ (_x\to\prec\xrightarrow{y} + \,_x\to\succ\xrightarrow{y})\succ\to_z = \xrightarrow{z\cdot (x\cdot y + y\cdot x)}=0, \]
	hence the relations $\alpha\prec(\beta\prec\gamma+\beta\succ\gamma)=0=(\alpha\prec\beta +\alpha\succ\beta)\succ\gamma$ in $(\op{Ass}_\circ(\sL eib),\prec,\succ)$. We leave to the reader to check that this is a generating set of relations, for instance by computing the Koszul dual $\op{Ass}_\bullet(\sZ inb)$ with the method of the previous section.

    As a final example, we consider $(\sZ inb,\cdot)$. In this case, we know form Example \ref{ex:leib} that $\op{Ass}_\circ(\sZ inb)=(\sA ss\bullet\sL eib)^!=di\sA ss^!=\sD end$: in fact, we get the dendriform relations on the operad $(\op{Ass}_\circ(\sZ inb),\prec,\succ)$, corresponding to the splittings as in Lemma \ref{lem:relwhite} (notice that the relation $(x\cdot y)\cdot z-(x\cdot z)\cdot y=0$ holds in $\sZ inb$)
    \[ _x\to\prec( \xrightarrow{y}\prec\to_z+\xrightarrow{y}\succ\to_z)-(_x\to\prec\xrightarrow{y})\prec\to_z=\xrightarrow{x\cdot(y\cdot z+ z\cdot y)-(x\cdot y)\cdot z}=0\]
    \[ (_x\to\prec\xrightarrow{y}+\,_x\to\succ\xrightarrow{y})\succ\to_z -\,_x\to\succ(\xrightarrow{y}\succ\to_z)=\xrightarrow{z\cdot(x\cdot y+y\cdot x)-(z\cdot y)\cdot x}=0\]
    \[ (_x\to\succ\xrightarrow{y})\prec\to_z -\, _x\to\succ(\xrightarrow{y}\prec\to_z)=\xrightarrow{(y\cdot x)\cdot z-(y\cdot z)\cdot x}=0.\]
\end{example}

\begin{theorem}\label{th:asswhite} There is a natural isomorphism $\op{Ass}_\circ(-)\xrightarrow{\cong}\sA ss\circ-$ of functors $\mathbf{Op}\to\mathbf{Op}$.
\end{theorem}
\begin{proof} We denote by $(\sA ss,\cup)$ the associative operad with its generating product. Given an operad $(\Oh,\cdot_i,\bullet_j,[-,-]_k)$ as usual, the operad $\sA ss\circ\Oh$ is generated by the tensor product operations (which are non-symmetric) $\cup\otimes\cdot_i, \cup\otimes\cdot_i^{op}$, $\cup\otimes\bullet_j$, $\cup\otimes[-,-]_k$, together with the larger set of relations holding in the tensor product $A\otimes V$ of a generic $\sA ss$-algebra $(A,\cup)$ and a generic $\Oh$-algebra $(V,\cdot_i,\bullet_j,[-,-]_k)$. The isomorphism of operads $\op{Ass}_\circ(\Oh)\leftrightarrow \sA ss\circ\Oh$ is given by $\prec_i\leftrightarrow \cup\otimes\cdot_i$, $\succ_i\leftrightarrow\cup\otimes\cdot_i^{op}$, $\circ_j\leftrightarrow\cup\otimes\bullet_j$ and finally $\ast_k\leftrightarrow\cup\otimes[-,-]_k$. This defines a morphism of operads $\sA ss\circ\Oh\to\op{Ass}_\circ(\Oh)$: in fact, by construction of the cup products $\prec_i,\succ_i,\circ_j,\ast_k$, these satisfy the relations of $\sA ss\circ\Oh$-algebra in the tensor product $C^*(\Delta_1;V)$ of the $\sA ss$-algebra $(C^*(\Delta_1;\mathbb{K}),\cup)$ (where we may forget the gradation) and the generic $\Oh$-algebra $V$. Conversely, we must show that if we evaluate a generating relation $R(-,-,-)=0$ of $\op{Ass}_\circ(\Oh)$, given as in Lemma \ref{lem:relwhite}, in the operations $\cup\otimes\cdot_i, \cup\otimes\cdot_i^{op}$, $\cup\otimes\bullet_j$, $\cup\otimes[-,-]_k$, we get a relation of the operad $\sA ss\circ\Oh$: but this is also clear, since given a generic $\sA ss$-algebra $(A,\cup)$ we find $R(a\otimes x,b\otimes y,c\otimes z)= (a\cup b\cup c)\otimes R'(x,y,z)=0$, $\forall a\otimes x,b\otimes y,c\otimes z\in A\otimes V$, where $a\cup b\cup c:=(a \cup b)\cup c=a\cup(b\cup c)$.
\end{proof}

The construction of the functor $\op{Perm}_\circ(-):\mathbf{Op}\to\mathbf{Op}$ is similar. Given a generic $\Oh$-algebra $V$, the cup products on the complex $C^*(\Delta_1,v_l;V)$ are defined by the same formulas as in the associative case.
\begin{definition} Given an operad $(\Oh,\cdot_i,\bullet_j,[-,-]_k)$ as usual, the operad $\op{Perm}_\circ(\Oh)$ is generated by non-symmetric operations $\prec_i,\succ_i,\circ_j,\ast_k$, together with the larger set of relations making $C^*(\Delta_1,v_l;V)$ with the cup products a local dg $\op{Perm}_\circ(\Oh)$-algebra for every $\Oh$-algebra $V$.\end{definition}

We want to describe a generating set of relations of $\op{Perm}_\circ(\Oh)$ as in the associative case. We denote the elements of the symmetric group $S_3=\{\id,(123),(132),(12),(13),(23)\}$ according to their cycle decomposition. Given a terniary operation $R(-,-,-):C^*(\Delta_1,v_l;V)^{\otimes 3}\to C^*(\Delta_1,v_l;V)$ in the cup products, we write $R=R_1+R_2+R_3$, where $R_1:=R_{\id}+R_{(23)}$, $R_2:=R_{(123)}+ R_{(12)}$ and $R_3:=R_{(132)}+R_{(13)}$. By Lemma \ref{lem:relpl}, the relation $R(-,-,-)=0$ holds in $C^*(\Delta_1,v_l;V)$ if and only if 
\[ R(\xrightarrow{x},\to_y,\to_z)=R(\to_x,\xrightarrow{y},\to_z)=R(\to_x,\to_y,\xrightarrow{z})=0\]
for all $x,y,z\in V$. We claim that this is true if and only if so is $R_i(-,-,-)=0$ for $i=1,2,3$: again, we limit ourselves to show $R(-,-,-)=0\solose R_1(-,-,-)=0$. This time, the point is that the only way to form a non-vanishing triple product out of $\xrightarrow{1},\to_1,\to_1\in (C^*(\Delta_1,v_l;\mathbb{K}),\cup)$ is by putting the degree one cochain $\xrightarrow{1}$ on the left: this implies $R_1(\to_x,\xrightarrow{y},\to_z)=R_1(\to_x,\to_y,\xrightarrow{z})=0$, whereas $R_1(\xrightarrow{x},\to_y,\to_z)=R(\xrightarrow{x},\to_y,\to_z)=\xrightarrow{R'(x,y,z)}=0$ by hypothesis, where $R'(-,-,-)=0$ is a relation of $\Oh$ since $V$ is generic, so we may conclude again by Lemma \ref{lem:relpl}. To sum up
\begin{lemma}\label{lem:relpm} The operad $\op{Perm}_\circ(\Oh)$ has a generating set of relations $R(-,-,-)=0$ of the form $R=R_{\id}+R_{(23)}$, consisting of all the possible splittings, as in the following 
	\[ R(\xrightarrow{x},\to_y,\to_z)=\xrightarrow{R'(x,y,z)}=0,\]
	of a relation $R'(-,-,-)=0$ in $\Oh$ via the cup products on $C^*(\Delta_1,v_l;V)$.\end{lemma}

\begin{theorem}\label{th:permwhite} There is a natural isomorphism $\op{Perm}_\circ(-)\xrightarrow{\cong}\sP erm\circ-$ of functors $\mathbf{Op}\to\mathbf{Op}$.
\end{theorem}
\begin{proof}
	This is done as in Theorem \ref{th:asswhite} by noticing that the relations of $\sP erm\circ\Oh$-algebra are satisfied by the cup products on the tensor product $C^*(\Delta_1,v_l;V)$ of the right permutative algebra $(C^*(\Delta_1,v_l;\mathbb{K}),\cup)$ and a generic $\Oh$-algebra $V$. Conversely, given a generating relation $R(-,-,-)=0$ of $\op{Perm}_\circ(\Oh)$ as in the previous lemma, it is straightforward to check that this holds more in general in the tensor product $A\otimes V$ of a generic $\sP erm$-algebra $A$ and a generic $\Oh$-algebra $V$. 
\end{proof}
\begin{example} We consider the operad $(\sL ie,[-,-])$: by the previous theorem $\op{Perm}_\circ(\sL ie)=\sP erm\circ\sL ie=\sL eib$, where the right Leibniz relation corresponds to the splitting as in Lemma \ref{lem:relpm}
	\[ (\xrightarrow{x}\ast \to_y)\ast\to_z-\xrightarrow{x}\ast(\to_y\ast\to_z)-(\xrightarrow{x}\ast\to_z)\ast\to_y=\xrightarrow{[[x,y],z]-[x,[y,z]]-[[x,z],y]}=0.\]

We consider the operad $(\sP ois,\bullet,[-,-])$. This is generated by a right permutative product $\circ$ and a right Leibniz product $\ast$: we have the following splittings of the Poisson identity
\begin{multline*}
0=\xrightarrow{[x\bullet y,z]-x\bullet [y,z]-[x,z]\bullet y} = \left( \xrightarrow{x}\circ\to_y\right)\ast\to_z -\xrightarrow{x}\circ(\to_y\ast\to_z)-\left( \xrightarrow{x}\ast\to_z\right)\circ\to{y}=\\= \left( \xrightarrow{x}\circ\to_y\right)\ast\to_z +\xrightarrow{x}\circ(\to_z\ast\to_y)-\left( \xrightarrow{x}\ast\to_z\right)\circ\to{y}, \end{multline*}
\begin{multline*} 0=\xrightarrow{[x,y\bullet z]-[x,y]\bullet z-[x,z]\bullet y}= \xrightarrow{x}\ast\left( \to_y\circ\to_z\right) -\left(\xrightarrow{x}\ast\to_y\right)\circ\to_z-\left(\xrightarrow{x}\ast\to_z\right)\circ\to_y =\\ = \xrightarrow{x}\ast\left( \to_z\circ\to_y\right) -\left(\xrightarrow{x}\ast\to_y\right)\circ\to_z-\left(\xrightarrow{x}\ast\to_z\right)\circ\to_y,\end{multline*}
corresponding to the three independent relations
\[ (\alpha\circ \beta)\ast\gamma=\alpha\circ(\beta\ast\gamma)+(\alpha\ast\gamma)\circ\beta,\quad\alpha\ast(\beta\circ\gamma)=(\alpha\ast\beta)\circ\gamma+(\alpha\ast\gamma)\circ\beta,\quad \alpha\circ(\beta\ast\gamma)=-\alpha\circ(\gamma\ast\beta), \]
in the operad $\op{Perm}_\circ(\sP ois)$. In fact, these are the relations defining Aguiar's operad $\sP erm\circ\sP ois=(pre\sL ie\bullet\sP ois)^!$ of dual pre-Poisson algebras \cite{A,U}.

We consider the operad $(\sA ss,\cdot)$: then we know $\sP erm\circ\sA ss=di\sA ss$. In this case, the relations we find as in Lemma \ref{lem:relpm} turn out to be the diassociative relations for the generating products $\prec$, $\curlyeqsucc:=\succ^{op}$ of $\op{Ass}_\circ(\Oh)$: in fact, we find
\begin{equation*} 0=\xrightarrow{(x\cdot y)\cdot z-x\cdot(y\cdot z)}=(\xrightarrow{x}\prec\to_y)\prec\to_z -\xrightarrow{x}\prec(\to_y\prec\to_z)=(\xrightarrow{x}\prec\to_y)\prec\to_z -\xrightarrow{x}\prec(\to_z\succ\to_y), \end{equation*}
\begin{equation*} 0=\xrightarrow{z\cdot(y\cdot x)-(z\cdot y)\cdot x}=(\xrightarrow{x}\succ\to_y)\succ\to_z -\xrightarrow{x}\succ(\to_y\succ\to_z)=\\=(\xrightarrow{x}\succ\to_y)\succ\to_z -\xrightarrow{x}\succ(\to_z\prec\to_y), \end{equation*}
\[ 0=\xrightarrow{(y\cdot x)\cdot z-y\cdot(x\cdot z)} =(\xrightarrow{x}\succ\to_y)\prec\to_x-(\xrightarrow{x}\prec\to_z)\succ\to_y.\]
The reader may compare this with the computation of $\op{preLie}_\bullet(\sA ss)$ in Example \ref{ex:ass}, as well as what we said in Remark \ref{rem:curlyeqprec}.

We consider the operad $(pre\sL ie,\cdot)$. We split the right pre-Lie identity as 
\[0=\xrightarrow{(x\cdot y)\cdot z-x\cdot(y\cdot z)-(x\cdot z)\cdot y +x\cdot(z\cdot y)}= (\xrightarrow{x}\prec\to_y)\prec\to_z-\xrightarrow{x}\prec\to_{y\cdot z}-(\xrightarrow{x}\prec\to_z)\prec\to_y+\xrightarrow{x}\prec\to_{z\cdot y},\]
Since we may split $\to_{y\cdot z}$ both as $\to_y\prec\to_z=\to_{y\cdot z}=\to_z\succ\to_y$, and similarly for $\to_{z\cdot y}$, we get the two independent relations  $(\alpha\prec\beta)\prec\gamma-\alpha\prec(\beta\prec\gamma)=(\alpha\prec\gamma)\prec\beta-\alpha\prec(\gamma\prec\beta)$ and $\alpha\prec(\beta\prec\gamma)=\alpha\prec(\gamma\succ\beta)$ in the operad $\op{Perm}_\circ(pre\sL ie)$. We also have the splitting
\[0=\xrightarrow{z\cdot (y\cdot x)-(z\cdot y)\cdot x-z\cdot(x\cdot y) +(z\cdot x)\cdot y}= (\xrightarrow{x}\succ\to_y)\succ\to_z-\xrightarrow{x}\succ\to_{z\cdot y}-(\xrightarrow{x}\prec\to_y)\succ\to_z+(\xrightarrow{x}\succ\to_{z})\prec\to_y,\]
giving the relations $(\alpha\succ\beta)\succ\gamma-\alpha\succ(\beta\succ\gamma)=(\alpha\prec\beta)\succ\gamma-(\alpha\succ\gamma)\prec\beta$ and $\alpha\succ(\beta\succ\gamma)=\alpha\succ(\gamma\prec\beta)$. We leave to the reader to check that this is a generating set of relations, and in fact the operad $(\op{Perm}_\circ(pre\sL ie),\prec,\curlyeqsucc:=\succ^{op})$ with the relations
\begin{eqnarray} \nonumber
(\alpha\prec\beta)\prec\gamma-\alpha\prec(\beta\prec\gamma)\:=\:(\alpha\prec\gamma)\prec\beta-\alpha\prec(\gamma\prec\beta),  & \alpha\prec(\beta\prec\gamma)\:=\:\alpha\prec(\beta\curlyeqsucc\gamma),  \\ \nonumber (\alpha\curlyeqsucc\beta)\curlyeqsucc\gamma-\alpha\curlyeqsucc(\beta\curlyeqsucc\gamma)\:=\:(\alpha\curlyeqsucc\gamma)\prec\beta-\alpha\curlyeqsucc(\gamma\prec\beta), & (\alpha\curlyeqsucc\beta)\curlyeqsucc\gamma \:=\:(\alpha\prec\beta)\curlyeqsucc\gamma,
\end{eqnarray}
is the Koszul dual of the operad $\op{preLie}_\bullet(\sP erm)$ from Example \ref{ex:perm}.
\end{example}

We finally come to the definition of $\op{Lie}_\circ(-):\mathbf{Op}\to\mathbf{Op}$. In this case, we consider the complex $C^*(\Delta_1;\mathbb{K})$ as a $\sL ie$-algebra via the bracket $[-,-]=\cup-\cup^{op}$, then given a generic $\Oh$-algebra $(V,\cdot_i,\bullet_j,\ast_k)$ we shall call the tensor product operations $\star_i := [-,-]\otimes\cdot_i=(\cup-\cup^{op})\otimes\cdot_i=\prec_i-\succ_i^{op}$, $\{-,-\}_j:=[-,-]\otimes\bullet_j=\circ_j-\circ_j^{op}$, $\circledast_k:=[-,-]\otimes[-,-]_k=\ast_k+\ast_k^{op}$ the cup brackets on $C^*(\Delta_1;V)$.
\begin{definition} Given an operad $(\Oh,\cdot_i,\bullet_j,[-,-]_k)$ as usual, the operad $\op{Lie}_\circ(\Oh)$ is generated by non-symmetric operations $\star_i$, anti-symmetric operations $\{-,-\}_j$ and symmetric operations $\circledast_k$ together with the larger set of relations making $C^*(\Delta_1;V)$ with the cup brackets a local dg $\op{Lie}_\circ(\Oh)$-algebra for every $\Oh$-algebra $V$.\end{definition}
\begin{theorem}\label{th:liewhite} There is a natural isomorphism $\op{Lie}_\circ(-)\xrightarrow{\cong}\sL ie\circ-$ of functors $\mathbf{Op}\to\mathbf{Op}$.
\end{theorem}
\begin{proof} This follows from Theorem \ref{th:asswhite}: in fact, by construction $\op{Lie}_\circ(\Oh)$ is the suboperad of $\op{Ass}_\circ(\Oh)$ generated by the operations $\star_i=\prec_i-\succ_i^{op}$, $\{-,-\}_j=\circ_j-\circ_j^{op}$, $\circledast_k=\ast_k+\ast_k^{op}$, and similarly $\sL ie\circ\Oh$ is the suboperad of $\sA ss\circ\Oh$ generated by the operations $(\cup-\cup^{op})\otimes\cdot_i$, $(\cup-\cup^{op})\otimes\bullet_j$, $(\cup-\cup^{op})\otimes[-,-]_k$.
\end{proof}

\begin{example}\label{ex:liewhite} The computations from the previous section, together with the above theorem and the fact that $\sL ie\circ\Oh=(\sC om\bullet\Oh^!)^!$, imply for instance $\op{Lie}_\circ(\sA ss)=\op{Lie}_\circ(pre\sL ie)=\sM ag_{1,0,0}$, $\op{Lie}_\circ(\sL ie)=\sM ag_{0,1,0}$, $\op{Lie}_\circ(\sP ois)=\sL ie+\sM ag_{0,1,0}$, $\op{Lie}_\circ(\sP erm)=\sL eib$, $\op{Lie}_\circ(\sZ inb)=pre\sL ie$. To illustrate the latter, we show that the (right) pre-Lie relation holds in the tensor product $L\otimes V$ of a generic Lie algebra $(L,[-,-])$ and a generic $\sZ inb$-algebra $(V,\cdot)$, equipped with the tensor product operation $\star:=[-,-]\otimes\cdot$. The associator of $\star$ is given by $A_\star(l\ten x,m\ten y,n\ten z)=[[l,m],n]\ten (x\cdot y)\cdot z\,-\,[l,[m,n]]\ten x\cdot(y\cdot z)$: to show that this is graded symmetric in the last two arguments, we compute
	\[ [[l,m],n]\ten (x\cdot y)\cdot z\,-\,[l,[m,n]]\ten x\cdot(y\cdot z)\,-\,[[l,n],m]\ten (x\cdot z)\cdot y\,+\,[l,[n,m]]\ten x\cdot(z\cdot y)  = \]
	\[ =[l,[m,n]]\ten\left( (x\cdot y)\cdot z-x\cdot(y\cdot z)-x\cdot(z\cdot y) \right)\,+\,[[l,n],m]\ten\left( (x\cdot y)\cdot z -(x\cdot z)\cdot y \right)\:=\:0. \] 

\end{example}

We are finally ready to complete the proof of Theorem \ref{th:black}: this will follow from the following theorem and theorems \ref{th:asswhite}, \ref{th:permwhite}, \ref{th:liewhite}.
\begin{theorem}\label{th:adj}  The following $\xymatrix{\op{Ass}_\bullet(-):\mathbf{Op}\ar@<2pt>[r]& \mathbf{Op}:\op{Ass}_\circ(-)\ar@<2pt>[l]}$, $\xymatrix{\op{Com}_\bullet(-):\mathbf{Op}\ar@<2pt>[r]& \mathbf{Op}:\op{Lie}_\circ(-)\ar@<2pt>[l]}$, $\xymatrix{\op{preLie}_\bullet(-):\mathbf{Op}\ar@<2pt>[r]& \mathbf{Op}:\op{Perm}_\circ(-)\ar@<2pt>[l]}$ are pairs of adjoint functors.
	\end{theorem}

\begin{proof} We consider the case of $\xymatrix{\op{Ass}_\bullet(-):\mathbf{Op}\ar@<2pt>[r]& \mathbf{Op}:\op{Ass}_\circ(-)\ar@<2pt>[l]}$ in detail.
	
So far we used overlapping notations for the generating sets of operations of $\op{Ass}_\bullet(\Oh)$ and $\op{Ass}_\circ(\Oh)$: this isn't practical to prove adjointness, thus, only for this proof, we will have to change notations; moreover, we will use slightly different generating sets than the ones we used before. We consider first the case of a non-symmetric operation $\cdot_i$ of $\Oh$. Corresponding to $\cdot_i$ there are two generating operations of $\op{Ass}_\bullet(\Oh)$ which we denote by $\underline{\cdot_i}$ and $\overline{\cdot_i}$ respectively: we use the generating set of operations from Remark \ref{rem:curlyeqprec}, explicitly, $\underline{\cdot_i}$ and $\overline{\cdot_i}$ are defined by the formulas $_x\to\cdot_i\xrightarrow{y}\:=\:\xrightarrow{x\underline{\cdot_i} y}\:=\:\xrightarrow{x}\cdot_i\to_y$ and $\to_x\cdot_i\xrightarrow{y}\:=\:\xrightarrow{x\overline{\cdot_i}y}\:=\:\xrightarrow{x}\cdot_i\,_y\to$ (with the previous notations from remark \ref{rem:curlyeqprec} we have $\underline{\cdot_i}=\prec_i$ and $\overline{\cdot_i}=\curlyeqsucc_i=-\succ_i^{op}$). Likewise, corresponding to $\cdot_i$ there are two generating operations of $\op{Ass}_\circ(\Oh)$ which we denote by $\cup\otimes\cdot_i$ and $\cup^{op}\otimes\cdot_i$ respectively: explicitly, these are defined by the formulas $_x\to(\cup\otimes\cdot_i)\xrightarrow{y}=\xrightarrow{x\cdot_i y}=\xrightarrow{x}(\cup\otimes\cdot_i)\to_y$ and $\to_x(\cup^{op}\otimes\cdot_i)\xrightarrow{y}=\xrightarrow{x\cdot_i y}=\xrightarrow{x}(\cup^{op}\otimes\cdot_i)\,_y\to$ (with the previous notations from this section we have $\cup\otimes\cdot_i=\prec_i$ and $\cup^{op}\otimes\cdot_i=\succ_i^{op}$). Finally, corresponding to $\cdot_i$ there are four generating operations of the operad $\op{Ass}_\circ(\op{Ass}_\bullet(\Oh))$, which we denote by $\cup\otimes\underline{\cdot_i}$, $\cup^{op}\otimes\underline{\cdot_i}$, $\cup\otimes\overline{\cdot_i}$ and $\cup^{op}\otimes\overline{\cdot_i}$ respectively, and four generating operations of $\op{Ass}_\bullet(\op{Ass}_\circ(\Oh))$, which we denote by $\underline{\cup\otimes\cdot_i}$, $\overline{\cup\otimes\cdot_i}$, $\underline{\cup^{op}\otimes\cdot_i}$ and $\overline{\cup^{op}\otimes\cdot_i}$ respectively. Similarly, to a generating symmetric operation $\bullet_j$ of $\Oh$ correspond a generating (non-symmetric) operation $\underline{\bullet_j}$ of $\op{Ass}_\bullet(\Oh)$ and a generating (non-symmetric) operation $\cup\otimes\bullet_j$ of $\op{Ass}_\circ(\Oh)$, as well as two generating (non-symmetric) operations $\cup\otimes\underline{\bullet_j}$, $\cup^{op}\otimes\underline{\bullet_j}$ of $\op{Ass}_\circ(\op{Ass}_\bullet(\Oh))$ and two generating (non-symmetric) operations $\underline{\cup\otimes\bullet_j}$, $\overline{\cup\otimes\bullet_j}$ of $\op{Ass}_\bullet(\op{Ass}_\circ(\Oh))$. Finally, to a generating antisymmetric operation $[-,-]_k$ of $\Oh$ correspond generating (non-symmetric) operations $\underline{[-,-]_k}$ and $\cup\otimes[-,-]_k$ of $\op{Ass}_\bullet(\Oh)$ and $\op{Ass}_\circ(\Oh)$ respectively, as well as generating operations $\cup\otimes\underline{[-,-]_k}$, $\cup^{op}\otimes\underline{[-,-]_k}$ of $\op{Ass}_\circ(\op{Ass}_\bullet(\Oh))$ and $\underline{\cup\otimes[-,-]_k}$, $\overline{\cup\otimes[-,-]_k}$ of $\op{Ass}_\bullet(\op{Ass}_\circ(\Oh))$ respectively.
	
Having established the previous notations, we will prove the theorem by explicitly exihibiting the unit $\varepsilon_\Oh:\Oh\to\op{Ass}_\circ(\op{Ass}_\bullet(\Oh))$ and the counit $\mu_{\Oh}:\op{Ass}_\bullet(\op{Ass}_\circ(\Oh))\to\Oh$ of the adjunction. The construction of the counit is implicit in the very definition of the functors $\op{Ass}_\circ(-)$ and $\op{Ass}_\bullet(-)$: given an $\Oh$-algebra $V$, the tensor product operations induce a local dg $\op{Ass}_\circ(\Oh)$-algebra structure on the complex of $V$-valued cochains $C^*(\Delta_1;V)=C^*(\Delta_1;\mathbb{K})\otimes V$, but by definition this is the same as an $\op{Ass}_\bullet(\op{Ass}_\circ(\Oh))$-algebra structure on the space $V$. Unraveling the definitions, we find that the counit $\mu_{\Oh}:\op{Ass}_\bullet(\op{Ass}_\circ(\Oh))\to\Oh$ is explicitly given by 
\[ \underline{\cup\otimes\cdot_i},\:\overline{\cup^{op}\otimes\cdot_i}\to\cdot_i,\qquad\overline{\cup\otimes\cdot_i},\: \underline{\cup^{op}\otimes\cdot_i}\to0, \] 
\[ \underline{\cup\otimes\bullet_j}\to\bullet_j,\qquad\overline{\cup\otimes\bullet_j}\to0,\qquad\underline{\cup\otimes[-,-]_k}\to[-,-]_k,\qquad\overline{\cup\otimes[-,-]_k}\to0. \]
As already remarked, this defines a morphism of operads $\mu_{\Oh}:\op{Ass}_\bullet(\op{Ass}_\circ(\Oh))\to\Oh$ by construction of the functors $\op{Ass}_\circ(-)$ and $\op{Ass}_\bullet(-)$. It remains to define the unit $\varepsilon_\Oh:\Oh\to\op{Ass}_\circ(\op{Ass}_\bullet(\Oh))$, this is explicitly given by
\[\cdot_i\to \cup\otimes\underline{\cdot_i}+\cup^{op}\otimes\overline{\cdot_i}, \qquad\bullet_j\to\cup\otimes\underline{\bullet_j}+(\cup\otimes\underline{\bullet_j})^{op},\qquad[-,-]_k\to\cup\otimes\underline{[-,-]_k}-(\cup\otimes\underline{[-,-]_k})^{op}. \]
We have to show that this is a morphism of operads. Given a generic $\op{Ass}_\bullet(\Oh)$-algebra $(V,\underline{\cdot_i},\overline{\cdot_i},\underline{\bullet_j},\underline{[-,-]_k})$, the complex $C^*(\Delta_1;V)$ carries both an $\op{Ass}_\circ(\op{Ass}_\bullet(\Oh))$-algebra structure via the tensor product operations and an $\Oh$-algebra structure $(C^*(\Delta_1;V),\cdot_i,\bullet_j,[-,-]_k)$ by definition of $\op{Ass}_\bullet(-)$. An easy verification shows $\alpha(\cup\otimes\underline{\cdot_i}+\cup^{op}\otimes\overline{\cdot_i})\beta=\alpha\cdot_i\beta$ for all $\alpha,\beta\in C^*(\Delta_1;V)$. For instance,
\[ \to_x (\cup\otimes\underline{\cdot_i}+\cup^{op}\otimes\overline{\cdot_i}) \xrightarrow{y}=\to_x(\cup^{op}\otimes\overline{\cdot_i})\xrightarrow{y}=\xrightarrow{x\overline{\cdot_i} y}=\to_x\cdot_i\xrightarrow{y}.\]
Similarly, one verifies $\alpha(\cup\otimes\underline{\bullet_j}+(\cup\otimes\underline{\bullet_j})^{op})\beta=\alpha\bullet_j\beta$ and $\alpha(\cup\otimes\underline{[-,-]_k}-(\cup\otimes\underline{[-,-]_k})^{op})\beta=[\alpha,\beta]_k$ for all $\alpha,\beta\in C^*(\Delta_1;V)$. Finally, given a relation $R(-,-,-)=0$ of $\Oh$, this induces a relation $R'(-,-,-)=0$ of $\op{Ass}_\bullet(\Oh)$ as in Lemma \ref{lem:rel}
\[  R(_x\to,\xrightarrow{y},\to_z)=\xrightarrow{R'(x,y,z)}=0.\]
On the other hand, by the previous considerations $R(-,-,-)$ in the left hand side can be computed equivalently either in the products $\cdot_i,\bullet_j,[-,-]_k$ of the $\Oh$-algebra structure or in the products $\cup\otimes\underline{\cdot_i}+\cup^{op}\otimes\overline{\cdot_i}, \:\:\cup\otimes\underline{\bullet_j}+(\cup\otimes\underline{\bullet_j})^{op},\:\:\cup\otimes\underline{[-,-]_k}-(\cup\otimes\underline{[-,-]_k})^{op}$ of the $\op{Ass}_\circ(\op{Ass}_\bullet(\Oh))$-algebra structure, which shows that $\varepsilon_\Oh$ sends an $\Oh$-algebra relation to an $\op{Ass}_\circ(\op{Ass}_\bullet(\Oh))$-algebra relation, and is thus a well defined morphism of operads.

To complete the proof, it remains to show that $\varepsilon_{-}$ and $\mu_{-}$ satisfy the conditions to be the unit and the counit of an adjunction. More precisely, we have to show that the compositions
\[ \op{Ass}_\bullet(\Oh)\xrightarrow{\op{Ass}_\bullet(\varepsilon_\Oh)}\op{Ass}_\bullet(\op{Ass}_\circ(\op{Ass}_\bullet(\Oh)))\xrightarrow{\mu_{\op{Ass}_\bullet(\Oh)}}\op{Ass}_\bullet(\Oh),  \]
\[ \op{Ass}_\circ(\Oh)\xrightarrow{\varepsilon_{\op{Ass}_\circ(\Oh)}}\op{Ass}_\circ(\op{Ass}_\bullet(\op{Ass}_\circ(\Oh)))\xrightarrow{\op{Ass}_\circ(\mu_\Oh)}\op{Ass}_\circ(\Oh),\]
are the respective identities. We consider the first one: notice that given a morphism $f:\Oh\to\sP$ of operads, the morhism $\op{Ass}_\bullet(f)$ is defined by $\op{Ass}_\bullet(f)(\underline{\cdot_i})=\underline{f(\cdot_i)}$, $\op{Ass}_\bullet(f)(\overline{\cdot_i})=\overline{f(\cdot_i)}$, $\op{Ass}_\bullet(f)(\underline{\bullet_j})=\underline{f(\bullet_j)}$ and $\op{Ass}_\bullet(f)(\underline{[-,-]_k})=\underline{f([-,-]_k)}$. Now it is easy to compute the first composition using the previous formulas (for a non-symmetric operation $\#$ we notice that $\underline{(\#^{op})}=(\overline{\#})^{op}$, in the following computation we apply this for $\#=\cup\otimes\underline{\bullet_j}$ and $\#=\cup\otimes\underline{[-,-]_k}$)
\[ \underline{\cdot_i}\to\underline{\cup\otimes\underline{\cdot_i}}+\underline{\cup^{op}\otimes\overline{\cdot_i}}\to\underline{\cdot_i}+0,\qquad\overline{\cdot_i}\to\overline{\cup\otimes\underline{\cdot_i}}+\overline{\cup^{op}\otimes\overline{\cdot_i}}\to0+\overline{\cdot_i},\]\[ \underline{\bullet_j}\to\underline{\cup\otimes\underline{\bullet_j}+(\cup\otimes\underline{\bullet_j})^{op}}=\underline{\cup\otimes\underline{\bullet_j}}+(\overline{\cup\otimes\underline{\bullet_j}})^{op}\to\underline{\bullet_j}+0,  \]\[ \underline{[-,-]_k}\to\underline{\cup\otimes\underline{[-,-]_k}-(\cup\otimes\underline{[-,-]_k})^{op}}=\underline{\cup\otimes\underline{[-,-]_k}}-(\overline{\cup\otimes\underline{[-,-]_k}})^{op}\to\underline{[-,-]_k}-0.  \]
Next we consider the second composition. We have $\op{Ass}_\circ(f)(\cup\otimes\cdot_i)=\cup\otimes f(\cdot_i)$, $\op{Ass}_\circ(f)(\cup^{op}\otimes \cdot_i)=\cup^{op}\otimes f(\cdot_i)$, and similarly for the other cases. As desired, the second composition is
\[\cup\otimes\cdot_i\to\cup\otimes\underline{\cup\otimes\cdot_i} +\cup^{op}\otimes\overline{\cup\otimes\cdot_i}\to\cup\otimes\cdot_i+0,\]\[ \cup^{op}\otimes\cdot_i\to\cup\otimes\underline{\cup^{op}\otimes\cdot_i} +\cup^{op}\otimes\overline{\cup^{op}\otimes\cdot_i}\to0+\cup^{op}\otimes\cdot_i,      \]
\[\cup\otimes\bullet_j\to\cup\otimes\underline{\cup\otimes\bullet_j} +\cup^{op}\otimes\overline{\cup\otimes\bullet_j}\to\cup\otimes\bullet_j+0,\]
\[ \cup\otimes[-,-]_k\to\cup\otimes\underline{\cup\otimes[-,-]_k} +\cup^{op}\otimes\overline{\cup\otimes[-,-]_k}\to\cup\otimes[-,-]_k+0.\] 

The remaining cases of $\xymatrix{\op{Com}_\bullet(-):\mathbf{Op}\ar@<2pt>[r]& \mathbf{Op}:\op{Lie}_\circ(-)\ar@<2pt>[l]}$ and  $\xymatrix{\op{preLie}_\bullet(-):\mathbf{Op}\ar@<2pt>[r]& \mathbf{Op}:\op{Perm}_\circ(-)\ar@<2pt>[l]}$ are proved by the same argument: in the former, we notice that we may identify $\op{Com}_\bullet(\op{Lie}_\circ(\Oh))$, $\Oh$ and $\op{Lie}_\circ(\op{Com}_\bullet(\Oh))$ with quotients of the same free operad $(\sM ag_{p,q,r},\cdot_i,\bullet_j,[-.-]_k)$ in such a way that the unit and the counit are the identities on the generating operations.\end{proof}

\end{document}